\documentclass[10pt,a4paper]{article}
\usepackage[utf8]{inputenc}
\usepackage{amsmath}
\usepackage{amsfonts}
\usepackage{color}
\usepackage{amssymb}
\usepackage{amsthm}
\usepackage{graphicx}
\usepackage{epstopdf}
\usepackage{caption}
\usepackage{subcaption}
\usepackage{a4,a4wide}
\usepackage{marginnote}
\usepackage{todo}
\definecolor{darkblue}{rgb}{0,0,0.6}
\definecolor{darkred}{rgb}{0.6,0,0}
\usepackage[colorlinks=true,urlcolor=darkblue,citecolor=darkblue,linkcolor=darkred]{hyperref}

\theoremstyle{plain}
\newtheorem{thm}{Theorem}[section]
\newtheorem{lem}[thm]{Lemma}
\newtheorem{prop}[thm]{Proposition}

\theoremstyle{definition}

\newtheorem{Rk}{Remark}

\newtheorem{ass}[thm]{\textbf{Assumption}}

\newcommand{\gr}[1]{{
#1}}
\newcommand{\qg}[1]{{
#1}}
\newcommand\delbyqg[1]{
}

\title{Existence and qualitative properties of travelling waves for an epidemiological model with mutations}

\author{Quentin Griette \footnote{Centre d'\'Ecologie Fonctionnelle et \'Evolutive, UMR
5175, CNRS, 1919 Route de Mende, 34293 Montpellier, France.
E-mail: quentin.griette@cefe.cnrs.fr}\,\,\footnote{I3M, Universit\'e Montpellier 2, place Eug\`ene Bataillon, 34095 Montpellier, France. E-mail quentin.griette@univ-montp2.fr}\:
Ga\"el Raoul 
\footnote{Centre d'\'Ecologie Fonctionnelle et \'Evolutive, UMR
5175, CNRS, 1919 Route de Mende, 34293 Montpellier, France.
E-mail: raoul@cefe.cnrs.fr.}
}

\begin{document}

\maketitle


\begin{abstract}
In this article, we are interested in a non-monotone system of logistic reaction-diffusion equations. This system of equations models an epidemics where two types of pathogens are competing, and a mutation can change one type into the other with a certain rate. We show the existence of minimal speed travelling waves, that are usually non monotonic. We then provide a description of the shape of those constructed travelling waves, and relate them to some Fisher-KPP fronts with non-minimal speed.
\end{abstract}

\section{Introduction}

\gr{Epidemics of newly emerged pathogen can have catastrophic consequences}. Among those who have infected humans, we can name the black plague, the Spanish flu, or more recently SARS, AIDS, bird flu or Ebola. Predicting the propagation of such epidemics is a great concern in public health. Evolutionary phenomena play an important role in the emergence of new epidemics\gr{: such epidemics typically start when the pathogen acquires the ability to reproduce in a new host, and to be transmitted within this new hosts population}. Another \gr{phenotype} that \gr{can} often vary rapidly is the virulence of the pathogen, that is how much the parasite is affecting its host; \gr{Field data show that the virulence of newly emerged pathogens changes rapidly, which moreover seems related to} unusual \gr{spatial dynamics} observed \gr{in such 
populations} (\cite{Hawley, Phillips}, see also \cite{Lion, Heilmann}). It is unfortunately difficult to set up experiments with a controlled environment to study evolutionary epidemiology phenomena with a spatial structure, we refer to \cite{Bell,Keymer} for current developments in this direction. Developing the theoretical approach for this type of problems is thus especially interesting. Notice finally that many current problems in evolutionary biology and ecology combine evolutionary phenomena and spatial dynamics: the effect of global changes on populations \cite{PY,DSE}, biological invasions \cite{SBP,KT}, cancers or infections \cite{Gerlingeretal,Frostetal}.

In the framework of evolutionary ecology, the virulence of a pathogen can be seen as a life-history trait of the pathogen \cite{Roff, Frank}. To explain and predict the evolution of virulence in a population of pathogens, many of the recent theories introduce a \textit{trade-off} hypothesis, namely a link between the parasite's virulence and its ability to transmit from one host to another, see e.g. \cite{Aletal}. The basic idea behind this hypothesis is that the more a pathogen reproduces (in order to transmit some descendants to other hosts), the more it ''exhausts'' its host. A high virulence can indeed even lead to the premature death of the host, which the parasite \gr{within this host} rarely survives.  In other words, by increasing its transmission rate, a pathogen reduces its own life expectancy. There exists then an optimal virulence trade-off, that may depend on the ecological environment. An environment that changes  in time \gr{(e.g. if the number of susceptible hosts is heterogeneous in 
time and/or space)} can then lead to a Darwinian evolution of the pathogen 
population. \gr{For 
instance,} in \cite{
Bernetal}, an experiment shows how the composition of a viral population (composed of the phage $\lambda$ and its virulent mutant $\lambda $cl857, which differs from $\lambda$ by a single locus mutation only) evolves in the early stages of the infection of an \textit{E. Coli} culture.

\medskip

The Fisher-KPP equation is a classical model for epidemics, and more generally for biological invasions, when no evolutionary phenomenon is considered. It describes the time evolution of the density $n=n(t,x)$ of a population, where $t\geqslant 0$ is the time variable, and $x\in\mathbb R$ is a space variable. The model writes as follows:
\begin{equation}\label{KPP}
 \partial_tn(t,x)-\sigma\Delta n(t,x)=rn(t,x)\left(1-\frac{n(t,x)}K\right).
\end{equation}
It this model, the term $\sigma\Delta n(t,x)=\sigma\Delta_x n(t,x)$ models the random motion of the individuals in space, while the right part of the equation models the logistic growth of the population (see \cite{Verhulst}): when the density of the population is low, there is little competition between individuals and the number of offsprings is then roughly proportional to the number of individuals, with a growth rate $r$\,; when the density of the population increases, the individuals compete for e.g. food, or in our case for susceptible hosts, and the growth rate of the population decreases, and becomes negative once the population's density exceeds the so-called carrying capacity $K$. The model \eqref{KPP} was introduced in \cite{fisher,KPP1937}, and the existence of travelling waves for this model, that is special solutions that describe the spatial propagation of the population, was proven in \cite{KPP1937}. Since then, travelling waves have had important implications in 
biology 
and physics, and raise many challenging problems. We refer to \cite{Xin} for an overview of this field of research.

\medskip

In this study, we want to model an epidemics, but also take into account the possible diversity of the pathogen population. It has been recently noticed that models based on \eqref{KPP} can be used to study this type of problems (see \cite{Betal,ACR,BM}). Following the experiment \cite{Bernetal} described above, we will consider two populations: a wild type population $w$, and a mutant population $m$. For each time $t\geqslant 0$, $w(t,\cdot)$ and $m(t,\cdot)$ are the densities of the respective populations over a one dimensional habitat $x\in\mathbb R$. The two populations differ by their growth rate in the absence of competition (denoted by $r$ in \eqref{KPP}) and their carrying capacity  (denoted by $K$ in \eqref{KPP}). We will assume that the mutant type is more virulent than the wild type, \gr{in the sense that} it will have an increased growth rate in the absence of competition (larger $r$), at the expense of a reduced carrying capacity (smaller $K$). We assume that the dispersal rate of the 
pathogen (denoted by $\sigma$ in \eqref{KPP}) is not affected by the mutations, and is then the same for the two types. Finally, when a parent gives birth to an offspring, a mutation occurs with a rate $\mu$, and the offspring will then be of a different type. Up to a rescaling, the model is then:

\begin{equation} \label{rescalled}
\left\{\begin{array}{l} 
           \partial_tw(t,x)-\Delta_x w(t,x)=w(t,x)\left(1-\left(w(t,x)+m(t,x)\right)\right)
             +\mu(m(t, x)-w(t, x)), \\
           \partial_tm(t,x)-\Delta_x m(t,x)=r\qg m(t,x)\left(1-\left(\frac{w(t,x)+m(t,x)}{K}\right)\right)
             +\mu(w(t,x)-m(t,x)),
           \end{array}\right. 
\end{equation}
where $t\geq 0$ is the time variable, $x\in\mathbb R$ is a spatial variable, $r>1$, $K<1$ and $\mu>0$ are constant coefficients. In \eqref{rescalled}, $r>1$ represents the fact that the mutant population reproduces faster than the wild type population if many susceptible hosts are available, while $K<1$ represents the fact that the wild type tends to out-compete the mutant if many hosts are infected. 
Our goal is to study the travelling wave solutions of \eqref{rescalled}, that is solutions with the 
following form : 
\[ w(t, x)=w(x-ct), \quad m(t, x)=m(x-ct), \]
with $ c\in \mathbb R$. \eqref{rescalled} can then be re-written as follows, with $x\in\mathbb R$:
\begin{equation}\label{systemefront}
 \left\{\begin{array}{l} 
           -cw'(x)-w''(x)=w(x)\left(1-\left(w(x)+m(x)\right)\right)
             +\mu(m(x)-w(x)), \\
           -cm'(x)-m''(x)=rm(x)\left(1-\left(\frac{w(x)+m(x)}{K}\right)\right)
             +\mu(w(x)-m(x)).
           \end{array}\right. 
\end{equation}
The existence of planar fronts in higher dimension ($x\in\mathbb R^N$) is actually equivalent to the $1D$ case ($x\in\mathbb R$), our analysis would then also be the first step towards the understanding of propagation phenomena for \eqref{rescalled} in higher dimension.

\medskip

\gr{There exists a large literature  on }travelling waves for systems of several interacting species. In some cases, the systems are monotonic (or can be transformed into a monotonic system). Then, sliding methods and comparison principles can be used, leading to methods close to the scalar case \cite{Volpert2,Volpert1,Roquejoffre}. The combination of the inter-specific competition and the mutations prevents the use of this type of methods here. 
 \gr{Other methods that have been used to study systems of interacting populations include phase plane methods (see e.g. \cite{Tang,Fei}) and singular perturbations (see \cite{Gardner2,Gardner}). }
 More recently, a different approach, based on a topological degree argument, has been developed for reaction-diffusion equations with non-local terms \cite{Nadin,ACR}. The method we use here to prove the existence of travelling wave for \eqref{systemefront} will indeed be derived from these methods. Notice finally that we consider here that dispersion, mutations and reproduction occur on the same time scale. This is an assumption that is important from a biological point of view (and which is satisfied in the particular $\lambda$ phage epidemics that guides our study, see \cite{Bernetal}). In particular, we will not use the Hamilton-Jacobi methods that have proven useful to study this kind of phenomena when different time scales are considered (see \cite{Mirrahimi, Betal, BM}).

\medskip

This mathematical study has been done jointly with a biology work, see \cite{GRG}. We refer to this other article for a deeper analysis on the biological aspects of this work, as well as a discussion of the impact of stochasticity for a related individual-based model (based on simulations and formal arguments).

\medskip

We will make the following assumption, 
\begin{ass}\label{ass}
$r\in (1,\infty)$, $\mu\in\left(0,\min\left(\frac{r}{2},1-\frac{1}{r},1-K,K\right)\right)$ and $K\in \left(0,\min\left(1,\frac{r}{r-1}\left(1-\frac{\mu}{1-\mu}\right)\right)\right)$.
\end{ass}

This assumption ensures the existence of a unique stationary solution of \eqref{rescalled} of the form $(w,m)(t,x)\equiv (w^*,m^*)\in(0,1)\times (0,K)$ (see Appendix~\ref{appendix_reaction_terms}). It does not seem very restrictive for biological applications, and we believe the first result of this study (Existence of travelling waves, Theorem~\ref{thm:main}) could be obtained under a weaker assumption, namely:
\[r\in (1,\infty),\quad K\in(0,1),\quad \mu\in (0,K).\]

%
%
%
%

\medskip

 Throughout this document we will denote by $f_w$ and $f_m$ the terms on the left hand side of \eqref{systemefront}:
\begin{equation}\label{def_f}
\begin{array}{l}
 f_w(w,m):=w(1-(w+m))+\mu(m-w), \\
f_m(w,m):=rm\left(1-\left(\frac{w+m}{K}\right)\right)+
                  \mu(w-m).
\end{array}
%
\end{equation}

We structure our paper as follows : in Section~\ref{section:main_results}, we will present the main results of this article, which are three fold: Theorem~\ref{thm:main} shows the existence of travelling waves for \eqref{systemefront}, Theorem~\ref{thm:monotonicity} describes the profile of the fronts previously constructed, and Theorem~\ref{thm:KPP} relates the travelling waves for \eqref{systemefront} to travelling waves of \eqref{KPP}, when $\mu$ and $K$ are small. sections~\ref{section:proof_box}, \ref{sec:monotonicity} and  \ref{sec:K_small} are devoted to the proof of the three theorems stated in Section~\ref{section:main_results}.


\section{Main results}\label{section:main_results}

The first result is the existence of travelling waves of minimal speed for the model \eqref{rescalled}, and an explicit formula for this minimal speed. We recall that the minimal speed travelling waves are often the biologically relevant propagation fronts, for a population initially present in a bounded region only (\cite{Bramson}), and it seems to be the one that is relevant when small stochastic perturbations are added to the model (\cite{MMQ}). Although we expect the existence of travelling waves for any speed higher than the minimal speed, we will not investigate this problem here - we refer to \cite{Nadin,ACR} for the construction of such higher speed travelling waves for related models. Notice also that the convergence of the solutions to the parabolic model \eqref{rescalled} towards travelling waves, and even the uniqueness of the travelling waves, remain open problems.

\begin{thm}\label{thm:main}
 Let $r,\,K,\,\mu$ satisfy Assumption~\ref{ass}. There exists a solution $(c,w,m)\in\mathbb R\times C^\infty(\mathbb R)^2$ of 
 \eqref{systemefront}, such that
\[ \forall x\in\mathbb R,\quad w(x)\in(0,1),\,m(x)\in (0,K), \]
\[ \underset{x\to-\infty}{\liminf}(w(x)+m(x))>0,\quad \underset{x\to\infty}{\lim}(w(x)+m(x))=0, \]
\[ c=c_*, \]
where 
\begin{equation}
 c_*:=\sqrt{2\left(1+r-2\mu+\sqrt{(r-1)^2+4\mu^2}\right)}
 \label{eq:defminc}
\end{equation}
is the minimal speed $c>0$ for which such a travelling wave exists.
\end{thm}
The difficulty of the proof of Theorem~\ref{thm:main} has several origins:
\begin{itemize}
 \item The system cannot be modified into a monotone system (see \cite{Tang,Calvez}), which prevents the use of sliding methods to show the existence of traveling waves.
\item The competition term has a negative sign, which means that comparison principles often cannot be used directly.
\end{itemize}

As mentioned in the introduction, new methods have been developed recently to show the existence of travelling wave in models with negative nonlocal terms (see \cite{Nadin, ACR}). To prove Theorem \ref{thm:main}, we take advantage of those recent progress by considering the competition term as a nonlocal term (over a set composed of only two elements : the wild and the virulent type viruses). The method of \cite{Nadin, ACR} are however based on the Harnack inequality \gr{(or related arguments)}, that are not as simple for systems of equations (see \cite{Busca}). We have thus introduced a different localized problem (see \eqref{pbnormbox}), which allowed us to prove our result without any Harnack-type argument.

\medskip

Our second result describes the shape of the travelling waves that we have constructed above. We show that three different shapes at most are possible, depending on the parameters. In the most biologically relevant case, where the mutation rate is small, we show that the travelling wave we have constructed in Theorem~\ref{thm:main} is as follows: the wild type density $w$ is decreasing, while the mutant type density $m$ has a unique global maximum, and is monotone away from this maximum. In numerical simulations of \eqref{rescalled}, we have always observed this situation (represented in Figure~\ref{fig-shape}), even for large $\mu$. This result also allows us to show that behind the epidemic front, the densities $w(x)$ and $m(x)$ of the two pathogens stabilize to $w^*$, $m^*$, which is the long-term equilibrium of the system if no spatial structure is considered. For some results on the monotony of solutions of the non-local Fisher-KPP equation, we refer to \cite{FZ,AC}. For models closer to \eqref{rescalled} (see e.g. \cite{ACR,Betal}), we do not believe any qualitative result describing the shape of the travelling waves exists.

\begin{figure}[h]
\centering
\includegraphics{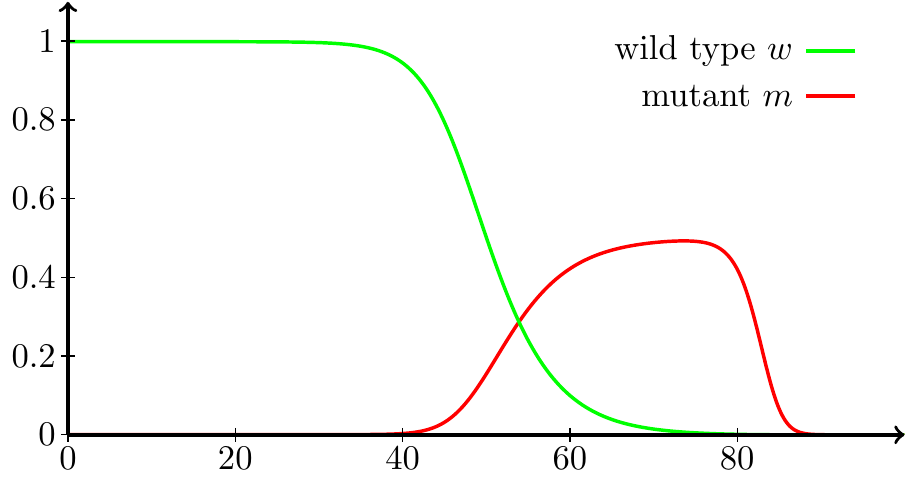}
\caption{Numerical simulation of \eqref{rescalled} with \qg{$r=2$, $K=0.5$, $\mu=0.01$}, with a heaviside initial condition for \qg{$w$ and null initial condition for $m$.} The numerical code is based on an \qg{implicit Euler scheme}. For large times, the solution seems to converge to a travelling wave, that we represent here, propagating towards large $x$. In the initial phase of the epidemics, the mutant ($m$, red line) population is dominant, but this mutant population is then quickly replaced by a population almost exclusively composed of wild types ($w$, green line).}
\label{fig-shape} 
\end{figure}

\begin{thm}\label{thm:monotonicity}
Let $r,\,K,\,\mu$ satisfy Assumption~\ref{ass}. There exists a solution $(c,w,m)\in \mathbb R_+\times C^\infty(\mathbb R)^2$ of \eqref{systemefront} such that 
\[\lim_{x\to-\infty}(w(x),m(x))=(w^*,m^*),\quad\lim_{x\to\infty}(w(x),m(x))=(0,0),\]
where $(w^*,m^*)$ is the only solution $(w^*,m^*)\in (0,1]\times (0,K]$ of $f_w(w,m)=f_m(w,m)=0$. 

The solution $(c,w,m)\in \mathbb R_+\times C^\infty(\mathbb R)^2$ satisfies one of the three following properties:
\begin{description}
 \item[(a)] $w$ is decreasing on $\mathbb R$, while $m$ is increasing on $(-\infty,\bar x]$ and decreasing on $[\bar x,\infty)$ for some $\bar x<0$, 
\item[(b)] $m$ is decreasing on $\mathbb R$, while $w$ is increasing on $(-\infty,\bar x]$ and decreasing on $[\bar x,\infty)$ for some $\bar x<0$,
\item[(c)] $w$ and $m$ are decreasing on $\mathbb R$.
\end{description}
Moreover, there exists $\mu_0=\mu_0(r,K)>0$ such that if $\mu<\mu_0$, then there exists a solution as above which satisfies $(\mathrm a)$.
\end{thm}

Finally, we consider the special case where the mutant population is small (due to a small carrying capacity $K>0$ of the mutant, and a mutation rate \gr{satisfying} $0<\mu<K$). If we neglect the mutants completely, the dynamics of the wild type would be described by the Fisher-KPP equation \eqref{KPP} (with $\sigma=r=K=1$), and they would then propagate at the minimal propagation speed of the Fisher-KPP equation, that is $c=2$. Thanks to Theorem~\ref{thm:main}, we know already that the mutant population will indeed have a major impact on the minimal speed of the population which becomes $c_*=2\sqrt{r}+\mathcal O(\mu)>2$, and thus shouldn't be neglected. In the next theorem, we show that the profile of $w$ is indeed close to the travelling wave of the Fisher-KPP equation with the non-minimal speed $2\sqrt r$, provided the conditions mentioned above are satisfied (see Figure~\ref{fig-KPP}). The effect of the mutant is then essentially to speed up the epidemics.

\begin{figure}[h]
\centering
\includegraphics[scale=0.75]{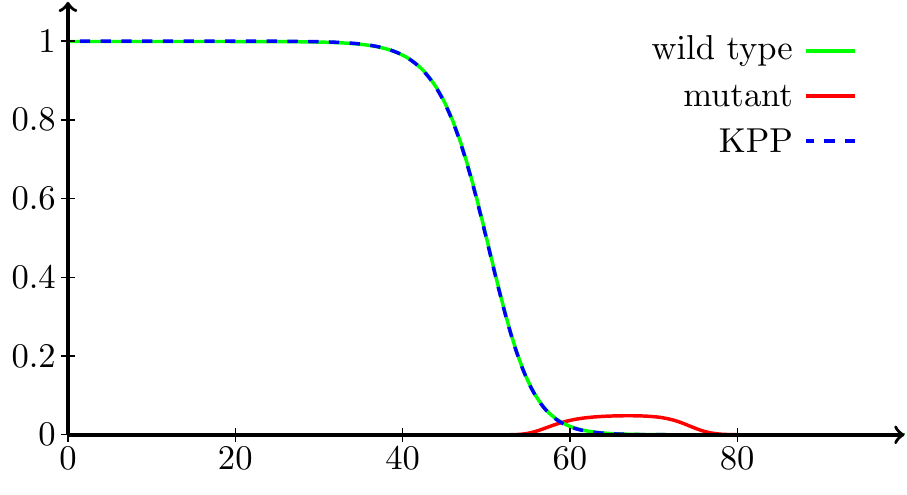}\hfill
\includegraphics[scale=0.75]{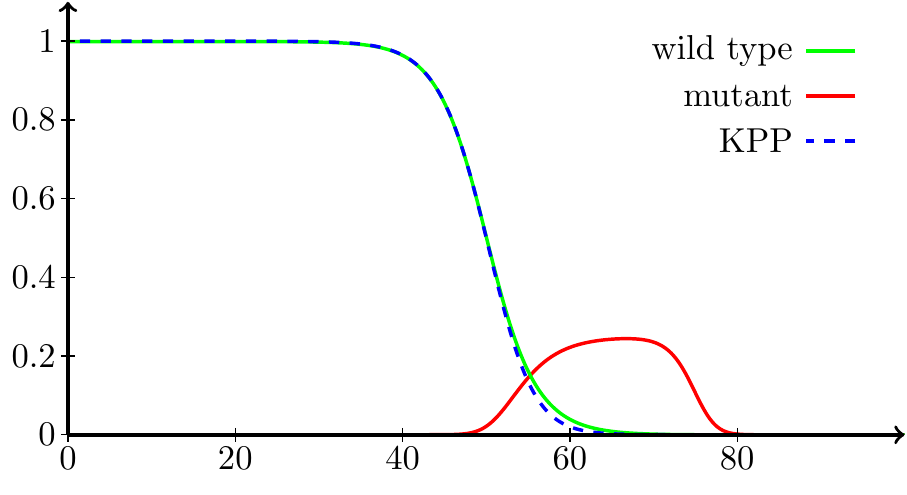}\\\medskip
\includegraphics[scale=0.75]{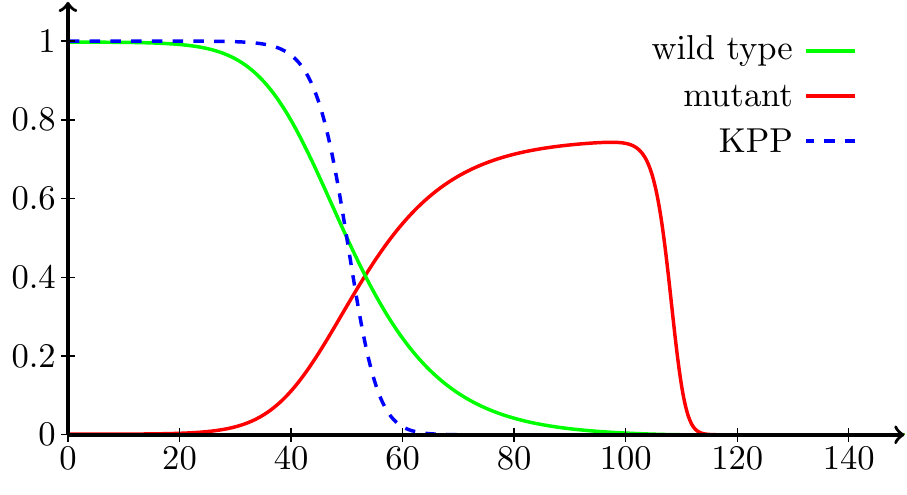}
\caption{Comparison of the travelling wave solutions of \eqref{rescalled} and the travelling wave solution of the Fisher-KPP equation of (non-minimal) speed $2\sqrt r$. These figures are obtained for \qg{$r=2$, $\mu=0.001$,} and three values of $K$: \qg{$K=0.05,\,0.25,\,0.75$.} We see that the agreement between the density of the wild type ($w$, green line) and the corresponding solution of the Fisher-KPP equation ($u$, dashed blue line) is good as soon as $K\leq 0.25$. The travelling waves solutions of \eqref{rescalled} are obtained numerically as long-time solutions of \eqref{rescalled} (based on an explicit Euler scheme), while the travelling waves solutions of the Fisher-KPP equations (for a the given speed $2\sqrt r$ that is not the minimal travelling speed for the Fisher-KPP model) is obtained thanks to a phase-plane approach, with a classical ODE numerical solver.}
\label{fig-KPP} 
\end{figure}

%
%
%
%
%
%
%
%
%

\begin{thm}\label{thm:KPP}
 Let $r\in(1,\infty)$, $K\in (0,1)$, $\mu\in(0,K)$ and $(c_*,w,m)\in\mathbb R\times C^0(\mathbb R)^2$ (see Theorem~\ref{thm:main} for the definition of $c_*$), $w,\,m>0$, a solution of \eqref{systemefront} such that 
\[ \underset{x\to-\infty}{\liminf}(w(x)+m(x))>0,\quad \underset{x\to\infty}{\lim}(w(x)+m(x))=0.\]
 There exists $ C=C(r)>0$, $ \beta\in\left(0, \frac{1}{2}\right) $ and $ \varepsilon>0 $ such that if $ 0<\mu<K<\varepsilon $, then
 \[ \Vert w-u\Vert_{L^\infty}\leqslant C K^\beta, \]
 where $ u \in C^0(\mathbb R)$ is a traveling wave of the Fisher-KPP equation, that is a solution (unique up to a translation) of
\begin{equation}\label{travelling_wave_KPP}
\left\{\begin{array}{l}
   -cu'- u''= u(1-u),\\
\underset{x\to-\infty}{\lim}u(x)=1,\; \underset{x\to\infty}{\lim} u(x)=0,
  \end{array}\right.
\end{equation}
 with speed $ c=\qg{c_0=2\sqrt r.} $
\end{thm}

 The Theorem~\ref{thm:KPP} is interesting from an epidemiological point of view: it describes a situation where the spatial dynamics of a population would be driven by the characteristics of the mutants, even though the population of these mutants pathogens is very small, and thus difficult to sample in the field.


\section{Proof of Theorem \ref{thm:main}}\label{section:proof_box}

We will prove Theorem~\ref{thm:main} in several steps. We refer to Remark~\ref{Rk:endproof} for the conclusion of the proof.

\subsection{A priori estimates on a localized problem}\label{subsection:apriori}

We consider first a restriction of the problem \eqref{systemefront} to a compact interval $[-a,a]$, for $a>0$. More precisely, we consider, for $c\in\mathbb R$,

\begin{equation}
\left\{\begin{array}{l}
w,\, m\in C^0([-a, a]),\\
-cw'-w''=f_w(w, m)\chi_{w\geq 0}\chi_{m\geq 0}, \\
-cm'-m''=f_m(w, m)\chi_{w\geq 0}\chi_{m\geq 0}, \\
w(-a)=w^*,\, m(-a)=m^*,\, w(a)=m(a)=0,
\end{array}\right. \label{eq:pbbox}
\end{equation}
where we have used the notation \eqref{def_f}, and $(w^*,m^*)$ are defined in the Appendix, see Subsection~\ref{appendix_reaction_terms}.

\subsubsection{Regularity estimates on solutions of \eqref{eq:pbbox}}\label{subsubsection:regularity}

The following result shows the regularity of the solutions of \eqref{eq:pbbox}.
\begin{prop}
Let $r,\,K,\,\mu$ satisfy Assumption~\ref{ass} and $a > 0$. If $ (w, m) \in L^\infty([-a, a]) $ satisfies
\begin{equation}
 \left\{\begin{array}{l} -cw'-w''=f_w(w, m), \\ -cm'-m''=f_m(w, m), \end{array}\right. 
\label{eq:thmregularity}
\end{equation}
on $ [-a, a]$, where $f_w,\,f_m$ are defined by \eqref{def_f}, and $c\in\mathbb R$, then $ w, m\in C^\infty([-a, a]). $ 
\label{thm:regularity}
\end{prop}
\begin{proof}[Proof of Proposition \ref{thm:regularity}]
 Since $f_w(w, m),\,f_m(w, m)\in L^\infty([-a,a])\subset L^p([-a, a])$ for any $p>1$, the classical theory (\cite{GT98}, 
 theorem 9.15) predicts that the solutions of the Dirichlet problem associated with  \eqref{eq:thmregularity} lies in \qg{$ W^{2, p}. $} This shows that $ w, m\in \qg{W^{2, p}} ((-a, a))$ for any $ p>1. $ But 
 then $ w, m\in C^{1, \alpha} ((-a, a))$ for any $ 0\leqslant \alpha < 1 $ (thanks to Sobolev embeddings). It follows
 that $ f(w, m) $ is a $ C^{1, \alpha}((-a, a)) $ function of the variable $ x\in(-a, a)$ (see \eqref{def_f} for the definition of $f$). Let us choose one such $ \alpha\in (0,1)$.  Now we can apply classical theory (\cite{GT98}, 
 theorem 6.14) to deduce that $ w, m\in C^{2, \alpha}((-a, a))$. But then $ w''$ and $m'' $ verify some uniformly 
 elliptic equation of the type
 \[ -c(w'')'-(w'')''=g, \]
\[ -c(m'')'-(m'')''=h, \]
with $ g, h\in C^{0, \alpha}((-a, a)) $, and we can apply again (\cite{GT98}, theorem 6.14). This argument can be used
recursively to show that $ w, m\in C^{2n, \alpha} ((-a, a))$ for any $ n\in\mathbb N $, so that finally, $ w, m\in C^\infty((-a, a))$.
\end{proof}

\subsubsection{Positivity and $ L^\infty $ bounds for solutions of \eqref{eq:pbbox}}\label{subsubsection:Linfty}

In this subsection, we prove the positivity of the solutions of \eqref{eq:pbbox}, as well as some $ L^\infty $ bounds.

\begin{prop}

Let $r,\,K,\,\mu$ satisfy Assumption~\ref{ass}, $a>0$, and $c\in\mathbb R$. If $(w,m)\in C^0([-a,a])^2$ is a solution of \eqref{eq:pbbox}, then $w$ and $m$ satisfy are positive, that is $w(x)>0$ and $m(x)>0$ for all $x\in [-a,a)$.
 \label{lem:samesolbox}
\end{prop}

\begin{proof}[Proof of Proposition \ref{lem:samesolbox}]
We observe that
  \[ f_w(w, m)=w(1-(w+m))+\mu(m-w)=w(1-\mu-w)+m(\mu-w), \]
  so that if $ w\leqslant\min(\mu, 1-\mu), $ then $ f_w(w, m)\chi_{w\geqslant 0}\chi_{m\geqslant 0}
  \geqslant 0. $ Let $ x_0\in [-a, a] $ such that $ w(x_0)\leqslant 0, $ and $ [\alpha, \beta] $ the 
  connex
  compound of the set $ \{w\leqslant \min(\mu, 1-\mu)\} $ that contains $ x_0. $ Since $ -cw'-w''\geqslant 0 $ 
  over $(\alpha, \beta) $ and $ w(\alpha), w(\beta)\geqslant 0, $ the weak minimum principle imposes $\underset{(\alpha, \beta)}{\inf}w\geqslant 0$, and thus $ w(x_0)=0. $ But then $ w $ reaches its global minimum at $ x_0, $ so the strong maximum
  principle imposes that $ x_0\in\{\alpha, \beta\}, $ or else $ w $ would be constant. We deduce then from 
  our hypothesis $ (w(-a)>0, w(a)=0) $ that $ x_0=\beta=a. $ That shows that $ w>0 $ in $ [-a, a). $ 
  
\medskip

To show that $m>0$, we notice that
\begin{eqnarray*}
f_m(w, m)&=&rm\left(1-\frac{w+m}{K}\right)+\mu(w-m) \\
  &=&m\left(r-\mu-\frac{r}{K}m\right)+  w\left(\mu-\frac{r}{K}m\right),
\end{eqnarray*}
  so that if $ m\leqslant\min\left(\frac{K}{r}\mu, K\left(1-\frac{\mu}{r}\right)\right), $ 
  then $ f_m(w, m)\chi_{w\geqslant 0}\chi_{m\geqslant 0}
  \geqslant 0$. The end of the argument to show the positivity of $w$ can the n be reproduced to show that $ m>0 $.

\end{proof}

\begin{prop}
Let $r,\,K,\,\mu$ satisfy Assumption~\ref{ass}, $a>0$, and $c\in\mathbb R$. If $(w,m)\in C^0([-a,a])^2$ is a positive solution of \eqref{eq:pbbox}, then $w$ and $m$ satisfy
\[ \forall x\in (-a, a), \quad w(x)< 1, \]
\[ \forall x\in (-a, a), \quad m(x)< K. \]
\label{thm:precisebound}
\end{prop}

\begin{proof}[Proof of Proposition \ref{thm:precisebound}]
 Let $ (w,m) $ a positive solution of \eqref{eq:pbbox}.
 \begin{itemize}
  \item We assume that there exists $ x_0\in(-a, a) $ such that $ w(x_0)>1. $ Let then $ [a_1, a_2] $ the
  connex compound of the set $ \{w\geqslant 1\} $ that contains $ x_0. $ Then in $ (a_1, a_2) $ we have
  \begin{eqnarray*}
   -cw'-w''&=&w(1-\mu-w-m)+\mu m\leqslant w(-\mu-m)+\mu m\\
   &=&m(\mu-w)-\mu w\leqslant 0,
  \end{eqnarray*}
  along with $ w(a_1)=w(a_2)=1, $ so that the weak {maximum} principle states $ w\leqslant 1 $ in 
  $ (a_1, a_2), $ which is absurd because $ w(x_0)>1. $ Therefore, $ w(x)\leqslant 1$ for all $x\in(-a, a)$  
  
  \item We assume that there exists $ x_0\in(-a, a) $ such that $ m(x_0)>K. $ Let then $ [a_1, a_2] $ the
  connex compound of the set $ \{m\geqslant K\} $ that contains $ x_0. $ Then in $ (a_1, a_2) $ we have
  \begin{eqnarray*}
   -cm'-m''&=&m\left(r-\mu-\frac{r}{K}(w+m)\right)+\mu w\leqslant m\left(-\mu-\frac{rw}{K}\right)+
  \mu w \\
  &=&w\left(\mu-\frac{r}{K}m\right)-\mu m \leqslant 0,
  \end{eqnarray*}
  Thanks to Assumption~\ref{ass}. Since $ m(a_1)=m(a_2)=K$, the weak {maximum} principle states $ m\leqslant K $ in 
  $ (a_1, a_2), $ which is absurd because $ m(x_0)>K. $ Therefore, $ m(x)\leqslant K$ for all $x\in(-a, a)$.
  
  \item  Now if $ w(x)\in(\max(\mu,1-\mu), 1]$, we still have the estimate 
  \[ -cw'(x)-w''(x)\leqslant m(x)(\mu-w(x))+w(x)(1-\mu-w(x))\leqslant 0,\]
  so that if there exists $x_0\in(-a,a)$ such that $ w(x_0)=1$, then $ w $ is locally equal to $ 1 $ thanks to the strong 
  maximum principle. But in that case
  \[ 0=(-cw'-w'')(x_0)=-m(x_0)+\mu(m(x_0)-1)<0, \]
  which is absurd. Hence, $w<1$.
  Similarly, if $ m(x_0)=K$, we get
  \[ 0=(-cm'-m'')(x_0)=-K\mu+w(x_0)(\mu-r)<0, \]
  which is absurd, and thus $m<K$.
  
 \end{itemize}
 
\end{proof}

\subsubsection{Estimates  on solutions of \eqref{eq:pbbox} when $c\geq c^*$ or $c=0$}\label{subsubsection:c}

The next result shows that the solutions of \eqref{eq:pbbox} degenerate when
$ a\rightarrow +\infty $ 
if the speed $ c$ is larger than a minimal speed $c^*$ (see Theorem \ref{thm:main} for the definition of $c^*$).

\begin{prop}[Upper bound on $c$]
Let $r,\,K,\,\mu$ satisfy Assumption~\ref{ass}. There exists $C>0$ such that for $a>0$ and $ c\geqslant c_* $, any solution  $(w,m)\in C^0([-a,a])^2$ of \eqref{eq:pbbox} satisfies
$$\forall x\in [-a,a],\quad \max\left(w(x), m(x)\right)\leq Ce^{\frac{-c-\sqrt{c^2-c_*^2}}{2}(x+a)}.$$
\label{lem:refupperbound}
\end{prop}

\begin{proof}[Proof of Proposition \ref{lem:refupperbound}]
 Let $ c\geqslant c_*$, and
\[ M:=\left(\begin{matrix} 1-\mu & \mu \\ \mu & r-\mu \end{matrix}\right) .\]
Since $M+\mu\, Id$ is a positive matrix, the Perron-Frobenius theorem implies that $M$ has a principal eigenvalue $h^+$ and a positive principal eigenvector $X$ (that is $X_i>0$ for $i=1,2$), given by
\begin{eqnarray}
 &h_+=\frac{1+r-2\mu+\sqrt{(1-r)^2+4\mu^2}}{2},\quad X=\left(\begin{matrix} 1-r+\sqrt{(1-r)^2+4\mu^2} \\ 2\mu \end{matrix}\right).&\label{def:hX}
\end{eqnarray}
The function $\psi_\eta(x):=\eta Xe^{\lambda_-x} $ with $\lambda_-:=\frac{-c-\sqrt{c^2-c_*^2}}{2}$ and $ \eta>0 $ is then a solution of the equation
\[ -c\psi_\eta'-\psi_\eta''=M\psi_\eta=h_+\psi_\eta. \]
We can define $ \mathcal A=\{\eta, (\psi_\eta)_1\geqslant w\textrm{ on }[-a,a]\}\cap\{\eta, (\psi_\eta)_2\geqslant m\textrm{ on }[-a,a]\} $, which is a closed subset of $ \mathbb R^+$. $\mathcal A$ is non-empty since $w$ and $m$ are bounded while $\left(Xe^{\lambda_-x}\right)_i \geq X_ie^{\lambda_-a}>0$ for $i=1,\,2$. 

Consider now $ \eta_0:=\inf\mathcal A$. Then $ (\psi_\eta)_1\geq w$, $ (\psi_\eta)_2\geq m$, and there exists $ x_0\in [-a,a] $ such that either $ (\psi_\eta)_1(x_0)=w(x_0) $ or $ (\psi_\eta)_2(x_0)=m(x_0)$. We first consider the case where 
$ (\psi_\eta)_1(x_0)=w(x_0)$. Then 
\[ -c(w-(\psi_\eta)_1)'(x_0)-(w-(\psi_\eta)_1)''(x_0)\leqslant-w(x_0)\left(w(x_0)+m(x_0)\right)\leqslant 0 \]
over $ [-a, a]$. The weak maximum principle (\cite{GT98}, theorem 8.1) implies that
\[ \underset{[-a, a]}{\sup}(w-(\psi_\eta)_1)=\max((w-(\psi_\eta)_1)(-a), (w-(\psi_\eta)_1)(a)), \]
and then, thanks to the definition of $ \eta_0$, $\underset{[-a, a]}{\sup}(w-(\psi_\eta)_1)=0$. Since $ w(a)=0<(\psi_\eta)_1(a), $ this means that $ (\psi_\eta)_1(-a)=w(-a) $, and thus
\[ \eta_0=\frac{b_w^-}{1-r+\sqrt{(1-r)^2+4\mu^2}}e^{\lambda_-a}. \]
The argument is similar if $ (\psi_\eta)_2(x_0)=m(x_0)$, which concludes the proof.
\end{proof}

%

The following Proposition will be used to show that $ c\neq 0$.
\begin{prop}
Let  $r,\,K,\,\mu$ satisfy Assumption~\ref{ass}, and $a >a_0:= \frac{\pi}{\sqrt{2(1-\mu)}}$. Every positive solution $ (w,m)\in C^0([-a,a])^2$ of \eqref{eq:pbbox} with $c=0$ satisfies the estimate
\begin{equation}\label{granden0} 
\underset{[-a_0, a_0]}{\max}(w+m)\geqslant \frac{K}{2}(1-\mu). 
 \end{equation}
\label{lem:nosmallc}
\end{prop}

\begin{proof}[Proof of Proposition \ref{lem:nosmallc}]
We assume that  $c=0$, $a>a_0$, and that \eqref{granden0} does not hold. We want to show that those assumptions lead to a contradiction.
For $ A\geqslant 0, $ the function defined by
\[ \psi_A(x)=A\cos\left(\sqrt{\frac{1-\mu}{2}}x\right), \]
is a solution of the equation $-\psi_A''=\frac{1-\mu}{2}\psi_A$ over $[-a_0,a_0]$. Since $ w,m>0 $ over $ [-a_0, a_0]$ and are bounded, the set  $\mathcal A:=\{A, \forall x\in[-a_0, a_0], \psi_A(x)\leqslant \min(w(x),m(x))\}$ is a closed bounded nonempty set in $(0, +\infty)$. 
Let now $ A_0 :=\max\mathcal A. $ 
We still have
 $ \psi_{A_0}\leqslant\min(w,m) $ over $ [-a_0, a_0]$, and then,  since \eqref{granden0} does not hold and $K<1$,
 \begin{eqnarray}
  -(w-\psi_{A_0})''&\geqslant& (1-\underset{[-a_0, a_0]}{\max}(w+m)-\mu)w-\frac{1-\mu}{2}\psi_{A_0}\label{eq1}\\
  &\geqslant&
\frac{1-\mu}{2}(w-\psi_{A_0})\geqslant 0.\nonumber
 \end{eqnarray}
Similarly, using additionally that $r>1$,
\[ -(m-\psi_{A_0})''\geqslant \frac{1-\mu}{2}(m-\psi_{A_0})\geqslant 0. \]
The weak {minimum} principle (\cite{GT98}, theorem 8.1) then imposes
\begin{align*}
& \min\left(\underset{[-a_0, a_0]}{\inf}(w-\psi_{A_0}),\underset{[-a_0, a_0]}{\inf}(m-\psi_{A_0})\right)\\
&\quad =\min((w-\psi_{A_0})(-a_0), (w-\psi_{A_0})(a_0), (m-\psi_{A_0})(-a_0), (m-\psi_{A_0})(a_0)).
\end{align*}
But the left side of the equation is $ 0 $ by definition of $ A_0, $ while the right side is strictly positive
since $ \psi_{A_0}(-a_0)=\psi_{A_0}(a_0)=0$. This contradiction shows the result.
\end{proof}

%
%

\begin{Rk}\label{Rk:gene_subsection_c}
 Notice that Propositions~\ref{thm:regularity}, \ref{lem:samesolbox},\ref{thm:precisebound}, \ref{lem:refupperbound} and \ref{lem:nosmallc} also holds if $(c,w,m)\in \mathbb R\times C^0([-a,a])$ is a solution of

\begin{equation}
\left\{\begin{array}{l}
w,\, m\in C^0([-a, a]),\\
-cw'-w''=\left(w(1-(w+\sigma m))+\mu(\sigma m-w)\right)\chi_{w\geq 0}\chi_{m\geq 0}, \\
-cm'-m''=\left(rm\left(1-\left(\frac{\sigma w+m}{K}\right)\right)+\mu(\sigma w-m) \right)\chi_{w\geq 0}\chi_{m\geq 0}, \\
w(-a)=w^*,\, m(-a)=m^*,\, w(a)=m(a)=0,
\end{array}\right. \label{eq:pbbox_sigma}
\end{equation}
where $\sigma\in[0,1]$.
\end{Rk}

\subsection{Existence of solutions to a localized problem}\label{subsection:existence_localized}

To show the existence of travelling waves solutions of \eqref{systemefront}, we will follow the approach of \cite{ACR}. The first step is to show the existence of solutions of \eqref{eq:pbbox} satisfying the additional normalization property $\underset{[-a_0, a_0]}{\max}(w+m)={\nu_0}$, that is the existence of a solution $(c,w, m)$ to
\begin{equation}
\left\{\begin{array}{l}
(c,w, m)\in\mathbb R\times C^0([-a, a])^2,  \\
-cw'-w''=f_w(w, m)\chi_{w\geqslant 0}\chi_{m\geqslant 0}, \\
-cm'-m''=f_m(w, m)\chi_{w\geqslant 0}\chi_{m\geqslant 0},  \\
w(-a)=w^*,\, m(-a)=m^*,\, w(a)=m(a)=0,\\
\underset{[-a_0, a_0]}{\max}(w+m)={\nu_0},
\end{array}\right.  \label{pbnormbox}
\end{equation}
where $f_w,\,f_m$ are defined by \eqref{def_f}, { $\nu_0=\min\left(\frac{K}{4}(1-\mu), \frac{w^*+m^*}{2}\right) $}  and $w^*,\,m^*$ are defined in Appendix \ref{appendix_reaction_terms}. 


 We introduce next the Banach space $(X,\Vert\cdot\Vert_X)$, with $ X:=\mathbb R\times C^0([-a, a])^2$ and $\Vert(c, w, m)\Vert_X := \max(|c|, \underset{[-a, a]}{\sup} |w|, \underset{[-a, a]}{\sup} |m|)$. We also define the operator 
\begin{equation}\label{Ksigma}
  \begin{array}{rccl}
    K^\sigma: & X & \longrightarrow & \qquad X, \\
         & (c,w,m) & \longmapsto & (c+\underset{[-a_0, a_0]}{\max}(\tilde w+\tilde m)-{\nu_0}, \tilde w, 
         \tilde m)
   \end{array} 
\end{equation}
where $ (\tilde w, \tilde m)\in C^0([-a, a])^2 $ is the unique 
solution of
\begin{equation*}
 \left\{\begin{array}{l} -c\tilde w'-\tilde w''= \left[w(1-(w+\sigma m))+\mu(\sigma m-w)\right]\chi_{w\geqslant 0}\chi_{m\geqslant 0}\textrm{ on } (-a, a), \\ 
-c\tilde m'-\tilde m''=\left[rm\left(1-\left(\frac{\sigma w+m}{K}\right)\right)+\mu(\sigma w-m)\right]\chi_{w\geqslant 0}\chi_{m\geqslant 0}  \textrm{ on } (-a, a), \\
\tilde w(-a)=w^*,\,\tilde m(-a)=m^*,\,  \tilde w(a)=\tilde m(a)=0.
\end{array}\right. 
\end{equation*}

The solutions of \eqref{pbnormbox} with $ c\geqslant 0 $  are then the fixed points of $ K^1 $ in the domain $\{(c, w, m), 0\leqslant w\leqslant 1, 0\leqslant m\leqslant K, c\geqslant 0\}$.

We define
\[ \Omega:=\left\{(c, w, m)\in \mathbb R_+\times C^0([-a,a])^2;\; c\in (0,c_*),\,\forall x\in[-a,a],\,-1<w(x)<1,\, -K<m(x)<K \right\},  \]
where $ c_* $ is defined by \eqref{eq:defminc}.

\begin{lem}\label{Lemma:continuous_operator_family}

Let  $r,\,K,\,\mu$ satisfy Assumption~\ref{ass}, and $a >0$. Then, $(K^\sigma)_{\sigma\in[0, 1]}$, 
defined by \eqref{Ksigma}, is a family of compact operators on $(X,\Vert\cdot\Vert_X)$, that is continuous with respect to 
$\sigma\in[0, 1]$.
\end{lem}
\begin{proof}[Proof of Lemma \ref{Lemma:continuous_operator_family}]
We can write $ K^\sigma=(\mathcal L_D)^{-1}\circ \mathcal F^\sigma $ where $ (\mathcal L_D)^{-1} $ is defined by
\[ (\mathcal L_D)^{-1}(c, g, h)=(\tilde c, \tilde w, \tilde m), \]
 where $ (\tilde c, \tilde w, \tilde m) $ is the unique solution of
\[ \left\{\begin{array}{l} -c\tilde w'-\tilde w''=g\textrm{ on } (-a, a), \\ 
                           -c\tilde m'-\tilde m''=h\textrm{ on }(-a, a), \\ 
                           \tilde w(-a)=w^*,\, \tilde m(-a)=m^*,\,\tilde w(a)=\tilde m(a)=0, \\
                           \tilde c=c+\underset{[-a_0, a_0]}{\max}(\tilde w+ \tilde m) - {\nu_0},
\end{array}\right. \]
and $ \mathcal F^\sigma $ is the mapping 
\[ \mathcal F^\sigma(c, w, m)=\left(c, w(1-(w+\sigma m))+\mu(\sigma m-w),rm\left(1-\frac{\sigma w+m}{K}\right)+\mu(\sigma w-m)\right). \]
$ \sigma\mapsto\mathcal F^\sigma$ is a continuous mapping from $ [0, 1] $ to $ C^0(\Omega, X)$, and $ (\mathcal L_D)^{-1} $ is a continuous application from $(X,\Vert\cdot\Vert_X)$ into itself (see Lemma \ref{lem:degreecompacity}), it then follows that $ \sigma\mapsto K^\sigma=(\mathcal L_D)^{-1}\circ \mathcal F^\sigma$ is a continuous mapping from $ [0, 1] $ to 
$ C^0(\Omega, X)$. Finally, the operator $(\mathcal L_D)^{-1}$ is compact (see Lemma \ref{lem:degreecompacity}), which implies that $ K^\sigma$ is compact for any fixed $ \sigma\in [0,1] $. 
\end{proof}

We now introduce the following operator, for $\sigma\in[0,1]$:
\begin{equation}\label{lemma:Ftau}
 F^\sigma:=Id-K^\sigma.
\end{equation}

Similarly, we introduce the operator
\begin{equation}\label{Ktau}
  \begin{array}{rccl}
    K_\tau: & X & \longrightarrow & \qquad X, \\
         & (c,w,m) & \longmapsto & (c+\underset{[-a_0, a_0]}{\max}(\tilde w+\tilde m)-{\nu_0}, \tilde w, 
         \tilde m)
   \end{array} 
\end{equation}
where $ (\tilde w, \tilde m)\in C^0([-a, a])^2 $ is the unique 
solution of
\begin{equation}\label{def_tilde_w_m} \left\{\begin{array}{l} -c\tilde w'-\tilde w''=\tau w(1-\mu -w)\chi_{w\geqslant 0}\chi_{m\geqslant 0} \textrm{ on } (-a, a), \\ 
-c\tilde m'-\tilde m''=\tau rm\left(1-\qg{\frac{\mu}{r}} -\frac mK\right) \chi_{w\geqslant 0}\chi_{m\geqslant 0} \textrm{ on } (-a, a), \\
\tilde w(-a)=w^*,\,\tilde m(-a)=m^*,\,  \tilde w(a)=\tilde m(a)=0.
\end{array}\right. 
\end{equation}
The argument of Lemma~\ref{Lemma:continuous_operator_family} can be be reproduced to prove that $(K_\tau)_{\tau\in[0,1]}$ is also a continuous family of compact operators on $(X,\|\cdot\|_X)$, and we can define, for $\tau\in[0,1]$, the operator
\begin{equation}\label{lemma:Ftau2}
 F_\tau:=Id-K_\tau.
\end{equation}

Finally, we introduce, for some $\bar c<0$ that we will define later on,
\[ \tilde\Omega:=\left\{(c, w, m)\in \mathbb R_+\times C^0([-a,a])^2;\; c\in (\bar c,c_*),\,\forall x\in[-a,a],\,-1<w(x)<1,\, -K<m(x)<K \right\}.\]

In the next Lemma, we will show that the Leray-Schauder degree of $F_0$ in the domain $\tilde \Omega$ is non-zero as soon as $a>0$ is large enough. We refer to chapter 12 of \cite{Smo} or to  chapter 10-11 of \cite{Brown} for more on the Leray-Schauder degree.
\begin{lem}
Let  $r,\,K,\,\mu$ satisfy Assumption~\ref{ass}. There exists $\bar a>0$ such that the Leray-Schauder degree of $F_0$ in the domain $\tilde \Omega$ is non-zero as soon as $a\geqslant \bar a$. 
\label{nonzerodegreeF0}
\end{lem}
\begin{proof}[Proof of Lemma \ref{nonzerodegreeF0}]

We first notice that for $\tau=0$, the solution $(\tilde w,\tilde m)$ of \eqref{def_tilde_w_m} is independent of $(w,m)$, and then, \[F_0(c,w,m)=\left({\nu_0}-\underset{[-a_0, a_0]}{\max}(w_{c}+m_{c}),w-w_c,m-m_c\right),\]
where $(w_{c},m_{c})$ is the solution of \eqref{def_tilde_w_m} with $\tau=0$, that is 
\[\left(w_c,m_c\right)(x):=\left(w^*\left(\frac{e^{-cx}-e^{-ca}}{e^{ca}-e^{-ca}}\right),m^*\left(\frac{e^{-cx}-e^{-ca}}{e^{ca}-e^{-ca}}\right)\right),\]
for $c\neq 0$, and $(w_c,m_c)(x)=(\frac{a-x}{2a}w^*,\frac{a-x}{2a}m^*)$ for $c=0$. 
The solutions of $F_0(c, w, m)=0$ then satisfy $w=w_c$ and $m=m_c$. In particular, the solutions of $F_0(c, w, m)=0$ satisfy $0<w<1$ and $0<m<K$ on $[-a,a)$, and then, 
\begin{eqnarray*}
 (c, w, m)&\notin&\left\{(\tilde c, \tilde w, \tilde m)\in \mathbb R\times C^0([-a,a])^2;\; \exists x\in[-a, a],\,  \tilde w(x)\in\{-1, 1\}\right\} \\
 &&\cup \left\{(\tilde c, \tilde w, \tilde m)\in \mathbb R\times C^0([-a,a])^2;\; \exists x\in[-a, a],\, \tilde m(x)\in\{-K, K\}\right\}\Big).
\end{eqnarray*}
The solutions of $F_0(c_*, w, m)=0$ also satisfy
\[\underset{[-a_0, a_0]}{\max}(w_{c_*}+m_{c_*})\leq 2\frac {e^{c_*a_0}}{e^{c_*a}-1},\]
so that $ \underset{[-a_0, a_0]}{\max}(w_{c_*}+m_{c_*})<
{\nu_0}$ if $a>\bar a$ for some $\bar a>0$. It follows that $F_0=0$ has no solution in $\overline{\tilde \Omega}\cap\left(\{c^*\}\times C^0([-a,a])^2\right)$, provided $a>\bar a$. 
Finally, for $c\leq 0$, the solutions of $F_0(c, w, m)=0$ satisfy $(w_c,m_c)(x)\geq (w_0,m_0)(x)=\left(-\frac{w^*}{2a}x+\frac{w^*}{2},-\frac{m^*}{2a}x+\frac{m^*}{2}\right)$, so that 
\[ \underset{[-a_0, a_0]}{\max}(w_c+m_c)>\underset{[-a_0, a_0]}{\max}(w_0+m_0)= \frac{w^*+m^*}{2}\left(1+\frac{a_0}{a}\right) > {\frac{w^*+m^*}{2}\geq \nu_0}, \]
and $F_0=0$ has no solution in $\overline{\tilde \Omega}\cap\left(\mathbb R_-\times C^0([-a,a])^2\right)$.


\medskip

We notice next that since $c\mapsto \underset{[-a_0, a_0]}{\max}(w_c+m_c)$ is decreasing, there exists a unique $c_0\in (0,c_*)$ such that $\underset{[-a_0, a_0]}{\max}(w_{c_0}+m_{c_0})={\nu_0}.$ We can then define
\[ \Phi_\tau(c, w, m)= \left({\nu_0}-\underset{[-a_0, a_0]}{\max}(w_c+m_c),
w-\left((1-\tau)w_c+\tau w_{c_0}\right) , m-\left((1-\tau)m_c+\tau m_{c_0}\right)\right),\]
which connects continuously $F_0=\Phi_0$ to 
\[ \Phi_1(c, w, m)=\left({\nu_0}-\underset{[-a_0, a_0]}{\max}(w_c+m_c),
w-w_{c_0},m-m_{c_0}\right). \]
Notice that $\Phi_\tau(c,w,m)=0$ implies $\underset{[-a_0, a_0]}{\max}(w_c+m_c)={\nu_0}$, which in turn implies that $c=c_0$. For any $\tau\in[0,1]$, the only solution of $ \Phi_\tau(c,w,m)=0 $ is then $(c_0,w_{c_0},m_{c_0})\not\in\partial\tilde \Omega$, which implies that the Leray-Schauder degree $\deg (F_0,\tilde \Omega)$ of $F_0$ is equal to $\deg(\Phi_1,\tilde\Omega)$, which can easily be computed since its variables are separated :
\begin{eqnarray*}
 \deg(\Phi_1,\Omega)&=&\deg\left({\nu_0}-\underset{[-a_0, a_0]}{\max}(w_c+m_c),(\bar c,c_*)\right)\\
&&\deg\left(w-w_{c_0},\left\{\tilde w\in C^0([-a,a]);\;-1<\tilde w(x)<1\right\}\right)\\
&&\deg\left(m-m_{c_0},\left\{\tilde m\in C^0([-a,a]);\;-K<\tilde m(x)<K\right\}\right)=1.
\end{eqnarray*}

\end{proof}

Next, we show that the Leray-Schauder degree of $F^0$ in the domain $\Omega$ is also non-zero, as soon as $a>0$ is large enough.

\begin{lem}
Let  $r,\,K,\,\mu$ satisfy Assumption~\ref{ass}. There exists $\bar a>0$ such that the Leray-Schauder degree of $F^0$ in the domain $\Omega$ is non-zero as soon as $a\geqslant \bar a$. 
\label{nonzerodegreeF0sigma}
\end{lem}
\begin{proof}[Proof of Lemma \ref{nonzerodegreeF0sigma}]
 Thanks to Proposition~\ref{lem:nosmallc} and Remark~\ref{Rk:gene_subsection_c}, any solution $(c,w,m)\in \tilde\Omega$ of \eqref{eq:pbbox_sigma}, and thus any solution $(c,w,m)\in \tilde\Omega$ of $F^0(c,w,m)=0$ satisfies $c>0$, that is $(c,w,m)\in\Omega$. Then,
\begin{equation}\label{eq2}
 \deg(F^0,\Omega)=\deg(F^0,\tilde\Omega)=\deg(F_1,\tilde\Omega).
\end{equation}
For $\tau\in[0,1]$, any solution $(c,w,m)\in\tilde \Omega$ of $F_\tau(c,w,m)=0$ satisfies 
\[-cw'-w''\geq -\mu w,\quad -cm'-m''\geq -\mu m,\]
and then $w,m\geq \phi_c$, where $\phi_c$ is the solution of $-c\phi_c'-\phi_c''= -\mu\phi_c$ with $\phi_c(-a)=K$, $\phi_c(a)=0$. This solution can easily be computed explicitly, and satisfies (for any fixed $a>0$)
\[\lim_{c\to-\infty}\phi_c(0)=K.\]
we can then choose $-\bar c>0$ large enough for $\phi_{\bar c}(0)\geq \nu_0$ to hold (note that the constant $\bar c\in\mathbb R$ is not independent of $a$). Then, $F_\tau(\bar c,w,m)=0$ implies $\underset{[-a_0, a_0]}{\max}(w+m)\geq 2\phi_{\bar c}(0)>\nu_0$, which implies in turn that $F_\tau(c,w,m)=0$ has no solution on $\left(\{\bar c\}\times C^0([-a,a])^2\right)\cap \overline{\tilde \Omega}$, for any $\tau\in[0,1]$. If $F_\tau(c,w,m)=0$ with $(c,w,m)\in\overline{\tilde \Omega}$, a classical application of the strong maximum principle shows that $0<w<1$ and $0<m<K$ on $(-a,a)$ (notice that $w$ and $m$ are indeed solutions of two uncoupled Fisher-KPP equations on $[-a,a]$). Moreover, the proof of Proposition~\ref{lem:refupperbound} applies to solutions of $F_\tau(c_*,w,m)=0$, which implies that (for any $\tau\in[0,1]$),
\[ \max_{[-a_0, a_0]}(w+m)\leq Ce^{-c_*\frac{a-a_0}2},\]
and thus, $F_\tau(c,w,m)=0$ has no solution on $\left(\{c_*\}\times C^0([-a,a])^2\right)\cap \overline{\tilde \Omega}$ as soon as $a>0$ is large enough (uniformly in $\tau\in[0,1]$).

We have shown that $F_\tau(c,w,m)=0$ has no solution on $\partial \tilde\Omega$ for $\tau\in[0,1]$. Since $\tau\mapsto F_\tau$ is a continuous familly of compact operators on $\tilde \Omega$, this implies that 
\[\deg(F_1,\tilde\Omega)=\deg(F_0,\tilde\Omega),\]
which, combined to \eqref{eq2} and Proposition~\ref{nonzerodegreeF0}, concludes the proof.

\end{proof}

\begin{prop}
Let  $r,\,K,\,\mu$ satisfy Assumption~\ref{ass}. There exists $\bar a>0$ such that for $a\geqslant \bar a$, there exists 
a solution $(c,w,m)\in \mathbb R\times C^0([-a,a])^2$ of \eqref{pbnormbox} with $c\in(0, c_*)$. 
\label{thm:localexistencedegree}
\end{prop}
\begin{proof}[Proof of Proposition \ref{thm:localexistencedegree}]

The first step of the proof is to show that there exists no solution  $(c,w,m)\in\partial\Omega$ of $F^\sigma(c,w,m)=0$ with $\sigma\in[0,1]$.

If such a solution exists, then Proposition \ref{lem:nosmallc} (see also Remark~\ref{Rk:gene_subsection_c}) implies that $c\neq 0$, and if $c=c_*$, then Proposition \ref{lem:refupperbound} (see also Remark~\ref{Rk:gene_subsection_c}) implies that
\begin{equation} 
\underset{[-a_0, a_0]}{\max}(w+ m)\leqslant Ce^{-c_*\frac{a-a_0}2},
\label{eq:prooflemdegreeFtau}
\end{equation}
where $C>0$ is a positive constant independent from $\sigma\in[0,1]$. If $a$ is large enough (more precisely if $a\geq a_0+\frac 2{c_*}\ln\left({\frac{2C}{\nu_0}}\right)$), then $\underset{[-a_0, a_0]}{\max}(w+m)\leq {\frac{\nu_0}{2}}$, which is a contradiction. Any solution $(c,w,m)\in \overline\Omega$ of $F^\sigma(c,w,m)=0$ then satisfies $c\in(0,c_*)$, as soon as $a>0$ is large enough.


Any solution $(c,w,m)\in \overline\Omega$ of $F^\sigma(c,w,m)=0$ is a solution of \eqref{eq:pbbox_sigma}, Proposition~\ref{lem:samesolbox} and Proposition~\ref{thm:precisebound} (see also Remark~\ref{Rk:gene_subsection_c}) then imply that for any $x\in(-a,a)$, $0< w(x)<1$ and $0< m(x)<K$.

We have shown that $F^\sigma(c,w,m)=0$ had no solution $(c,w,m)\in\partial\Omega$, for $\sigma\in [0,1]$. Since moreover $(F_\sigma)_{\sigma\in[0,1]}$ is a continuous family of compact operators (see Lemma~\ref{Lemma:continuous_operator_family}) this is enough to show that $\deg(F^\sigma,\Omega)$ is independent of $\sigma\in[0,1]$, and then, thanks to Lemma~\ref{nonzerodegreeF0sigma}, as soon as $a>0$ is large enough,
\[\deg(F^1,\Omega)=\deg(F^0,\Omega)\neq 0.\]
which implies in particular that there exists at least one solution $(c,w,m)\in\Omega$ of $F^1(c,w,m)=0$, that is a solution $(c,w,m)$ of \eqref{pbnormbox} in $\Omega$.

%
%
%
\end{proof}


\subsection{Construction of a travelling wave}\label{subsection:whole_line}

\begin{prop}\label{prop:existence_front}
Let  $r,\,K,\,\mu$ satisfy Assumption~\ref{ass}. There exists a solution $(c, w, m)\in(0, c_*]\times C^0(\mathbb R)^2$ of problem \eqref{systemefront} that satisfies $0<w(x)<1$ and $0<m(x)<K$ for $x\in\mathbb R$, as well as $(w+m)(0)={\nu_0}. $
\end{prop}
\begin{proof}[Proof of Proposition \ref{prop:existence_front}]
 For $n\geq 0$, let $ a_n:=\bar a+n $ (where $ \bar a>0 $ is defined in Proposition \ref{thm:localexistencedegree}), and $(c_n, w_n, m_n)$ a solution of \eqref{pbnormbox} provided by Proposition \ref{thm:localexistencedegree}. 
%
We denote by $(w_n^k, m_n^k)$ the restriction  of $(w_n, m_n)$ to $ [-a_k, a_k] $ $ (k<n). $ From interior elliptic estimates (see e.g. Theorem 8.32 in \cite{GT98}), we know that there exists a constant $C>0$ independent of $k>0$, such that for any $n\geq k+1$,
\[ \max\left(\left\Vert w_n|_{[-a_k, a_k]}\right\Vert_{C^1([-a_k, a_k])},\left\Vert m_n|_{[-a_k, a_k]}\right\Vert_{C^1([-a_k, a_k])}\right)\leqslant C, \]
 
 Since $c_n\in [0,c^*]$ for all $n\in\mathbb N$, we can extract from $(c_n, w_n, m_n)$ a subsequence (that we also denote by $(c_n, w_n, m_n)$), such that $c_n\rightarrow c_0$ for some $c_0\in[0,c^*]$. Since $c_n\in (0,c_*)$ for all $n\geq 3$, the limit speed satisfies $c_0\in[0,c_*]$. Thanks to Ascoli's Theorem, $C^1([-a_k, a_k])$ is compactly embedded in $C^0([-a_k,a_k])$. We can then use a diagonal extraction, to get a subsequence such that $w_n$ and $m_n$ both converge uniformly on every compact interval of $\mathbb R$. 
Let $w_0,\, m_0\in C^0(\mathbb R)$ the limits of $(w_n)_n$ and $(m_n)_n$ respectively. Then, thanks to the uniform convergence, we get that
 \[\forall x\in\mathbb R,\quad 0\leqslant w_0(x)\leqslant 1,\quad 0\leqslant m_0(x)\leqslant K, \]
 \[ -c_0w_0'-w_0''=f_w(w_0, m_0)\textrm{ on }\mathbb R, \]
 \[ -c_0m_0'-m_0''=f_m(w_0, m_0)\textrm{ on }\mathbb R, \]
 in the sense of distributions. Thanks to Proposition \ref{thm:regularity}, these two functions are smooth and are thus classical solutions of \eqref{systemefront}. Moreover,  $ \underset{[-a_0, a_0]}{\max}(w_0+m_0)={\nu_0}$, and Lemma \ref{lem:nosmallc} implies that $c_0\neq 0$. Finally, up to a shift, $w_0(0)+m_0(0)={\nu_0}$.

\end{proof}


In the next proposition, we show that the solution of \eqref{systemefront} obtained in Proposition~\ref{prop:existence_front} are indeed propagation fronts.

\begin{prop}
Let  $r,\,K,\,\mu$ satisfy Assumption~\ref{ass} and $(c, w, m)\in\mathbb R\times C^0(\mathbb R)^2$ a solution of \eqref{systemefront} such that $(w+m)(0)={\nu_0}$. 
 Then $w+m$ is decreasing on $(0,+\infty)$, 
\[\lim_{x\to \infty}w(x)=\lim_{x\to \infty}m(x)=0,\]
and $w(x)+m(x)\geq {\nu_0}$ on $(-\infty,0]$.
 \label{prop:propright}
\end{prop}

\begin{proof}[Proof of Proposition \ref{prop:propright}]

Assume that $w(x)+m(x)< K$, and $w'(x)+m'(x)\geq 0$. Then, \qg{
\begin{equation}\label{eqconcavity}
 -c(w+m)'(x)-(w+m)''(x)=w(x)(1-(w(x)+m(x)))+rm(x)\left(1-\frac{w(x)+m(x)}{K}\right),
\end{equation} } 
\qg{with the right side positive, }and then $(w+m)''(x)< 0$.
%
If there exists $x_0\in\mathbb R$ satisfying $w(x_0)+m(x_0)< K$, and $w'(x_0)+m'(x_0)\geq0$, then we can define $ \mathcal C=\{x\leqslant x_0, \forall y\in[x, x_0], (w+m)''(y)\leqslant 0\}$. Then $\mathcal C\neq\varnothing$ and $\mathcal C$ is closed since $(w+m)''$ is continuous. Let $x_1\in \mathcal C$. Then $(w+m)'$ is decreasing on $[x_1, x_0]$, so that $(w+m)'(x_1)\geqslant (w+m)'(x_0)>0$ and $(w+m)(x_1)\leqslant (w+m)(x_0)$. \eqref{eqconcavity} then implies that $(w+m)''(x_1)<0$, which proves that $ \mathcal C $ is open, and thus $\mathcal C=(-\infty, 0)$. This implies in particular that $w(x)+m(x)<0$ for some $x<x_0$, which is a contradiction. We have then proven that $x\mapsto w(x)+m(x)$ is decreasing on $[x_0,\infty)$ as soon as $w(x_0)+m(x_0)\leq K$. It implies that $w(x)+m(x)\geq {\nu_0}$ for $x\leq 0$, and that $x\mapsto w(x)+m(x)$ is decreasing on $[0,\infty)$.

Then, $\lim_{x\to\infty}w(x)+m(x)=l\in [0,K)$ exists, which implies that $\lim_{x\to\infty}w'(x)+m'(x)=\lim_{x\to\infty}w''(x)+m''(x)=0$, since $w$ and $m$ are regular. Then,
\[\lim_{x\to\infty}\left( -c(w+m)'(x)-(w+m)''(x)\right)=0,\]
which, combined to \eqref{eqconcavity}, proves that $\lim_{x\to\infty}w(x)+m(x)=0$. 

\end{proof}

\subsection{Characterization of the speed of the constructed travelling wave}

\begin{lem}
Let  $r,\,K,\,\mu$ satisfy Assumption~\ref{ass} and $(c,w,m)\in\mathbb R\times C^0(\mathbb R)^2$ a solution of \eqref{systemefront} such that $(w+m)(0)={\nu_0}$. Then there exists $x_0\in\mathbb R$ and $C>0$ such that
\[\forall x\geqslant x_0,\quad w(x)+m(x)\leqslant C \min(w(x),m(x)).\]
\label{lem:compuvsum}
\end{lem}

\begin{proof}[Proof of Lemma~\ref{lem:compuvsum}]
 Let $S(x):=w(x)+m(x)$, and $\alpha>0$. Then
\[-c(S-\alpha w)'-(S-\alpha w)''=((1-S)-(1-S-\mu)\alpha)w+(S-w)\left(r\left(1-\frac{S}{K}\right)-\alpha\mu\right). \]
Let $x_1\in\mathbb R$ such that $S(x)\leqslant \frac{1-\mu}{2}$ for all $x\geq x_1$ ($x_1$ exists thanks to Proposition~\ref{prop:propright}). Then, for $\alpha\geqslant\alpha_0:=\max\left(\frac{2}{1-\mu}, \frac{r}{\mu}\right)$ and $x\geqslant x_1$,
\[-c(S-\alpha w)'(x)-(S-\alpha w)''(x)\leqslant 0. \]
Let $\alpha_1:=\max\left(\frac{S(x_1)}{w(x_1)}, \alpha_0\right)+1$. Then $-c(S-\alpha_1 w)'-(S-\alpha_1 w)''\leqslant 0 $ over $(x_1, +\infty)$. We can then apply the  weak maximum principle to show that for any $x_2>x_1$,
\[\max_{[x_1,x_2]} (S-\alpha_1 w)(x)=\max\left((S-\alpha_1 w)(x_1),(S-\alpha_1 w)(x_2)\right).\]
Since $(S-\alpha_1 w)(x_1)\leqslant 0$ and $\underset{x_2\rightarrow+\infty}{\lim}(S-\alpha_1 w)(x_2)=0$, we have indeed shown that $\sup_{[x_1,\infty)} (S-\alpha_1 w)(x)=0$, and then,
\[\forall x\geq x_1,\quad w(x)+m(x)\leq \alpha_1w(x).\]
A similar argument can be used to show that there exists $x_2\in\mathbb R$ and $\alpha_2>0$ such that $w(x)+m(x)\leqslant \alpha_2 m(x)$ for $x\geqslant x_2$, which concludes the proof of the Lemma.
\end{proof}

\begin{prop}\label{prop:vitesse_front}
Let  $r,\,K,\,\mu$ satisfy Assumption~\ref{ass} and $(c,w,m)\in\mathbb R_+\times C^0(\mathbb R)^2$ a solution of \eqref{systemefront} such that $(w+m)(0)={\nu_0}$ and $c\leq c_*$. Then, $c=c_*$.
\end{prop}

\begin{Rk}\label{Rk:endproof}
 Combined to Proposition~\ref{prop:existence_front} and Proposition~\ref{prop:propright}, this proposition completes the proof of Theorem~\ref{thm:main}.
\end{Rk}

\begin{proof}[Proof of Proposition \ref{prop:vitesse_front}]

Let $(c,w,m)\in[0,c_*]\times C^0(\mathbb R)^2$ a solution of \eqref{systemefront} such that $(w+m)(0)={\nu_0}$. Thanks to Lemma~\ref{lem:compuvsum}, there exists $x_0>0$ and $C>0$ such that
\begin{equation*}
 \left\{\begin{array}{l}
         -cw'-w''\geq w(1-\mu-C w)+\mu m,\\
	 -cm'-m''\geq m(r-\mu-C m)+\mu w.
        \end{array}
\right.
\end{equation*}
Let now $\varphi_\eta(x+x_1):=\eta e^{-\frac c2x}\sin\left(\frac{\sqrt{4h-c^2}}2x\right)$, where $\eta>0$, $h\geq c^2/4$ and $x_1>x_0$. $\varphi_\eta$ then satisfies $\varphi_\eta(x_1)=\varphi_\eta\left(x_1+2\pi/\sqrt{4h-c^2}\right)=0$, and $-c\varphi_\eta'-\varphi_\eta''=h\varphi_\eta$ on $\left[x_1,x_1+2\pi/\sqrt{4h-c^2}\right]$. $\psi_\eta:=\varphi_\eta X$ ($X$ is defined by \eqref{def:hX}) is then a solution of
\[-c\psi_\eta'-\psi_\eta''=(M+(h-h_+)Id)\psi_\eta,\]
where $h_+$ is defined by \eqref{def:hX}, and we can also write this equality as follows
\begin{equation*}
 \left\{\begin{array}{l}
         -c(\psi_\eta)_1'-(\psi_\eta)_1''= (\psi_\eta)_1\left(1-\mu+(h-h_+)\right)+\mu (\psi_\eta)_2,\\
	 -c(\psi_\eta)_2'-(\psi_\eta)_2''= (\psi_\eta)_2\left(r-\mu+(h-h_+)\right)+\mu (\psi_\eta)_1.
        \end{array}
\right.
\end{equation*}
Assume now that $c<c_*$. Then, we can choose $c^2/4<h<c_*^2/4=h_+$, and define 
\[\bar\eta=\max\left\{\eta>0;\,\forall x\in\left[x_1,x_1+2\pi/\sqrt{4h-c^2}\right],\, (\psi_\eta)_1(x)\leq w(x),\,(\psi_\eta)_2(x)\leq m(x)\right\}.\]
Since $w$ and $m$ are positive bounded function, $\bar\eta>0$ exists, and since 
$$(\psi_\eta)_i(x_1)=(\psi_\eta)_i\left(x_1+2\pi/\sqrt{4h-c^2}\right)=0,$$ 
there exists $\bar x\in \left(x_1,x_1+2\pi/\sqrt{4h-c^2}\right)$ such that either $(\psi_\eta)_1(\bar x)= w(\bar x)$ or $(\psi_\eta)_2(\bar x)= m(\bar x)$. Assume w.l.o.g. that $(\psi_\eta)_1(\bar x)= w(\bar x)$. Then $w-(\psi_\eta)_1$ has a local minimum in $\bar x$, which implies that
\begin{eqnarray*}
 0&\geq& -c(w-(\psi_{\bar \eta})_1)'(\bar x)-(w-(\psi_{\bar \eta})_1)''(\bar x)\\
&\geq& \left[w(\bar x)(1-\mu-C w(\bar x))+\mu m(\bar x)\right]-\left[(\psi_{\bar \eta})_1(\bar x)\left(1-\mu+(h-h_+)\right)+\mu (\psi_{\bar \eta})_2(\bar x)\right]\\
&\geq&(\psi_{\bar \eta})_1(\bar x)\left[(h_+-h)-C(\psi_{\bar \eta})_1(\bar x)\right],
\end{eqnarray*}
and then $\bar \eta\geq (h_+-h)/(CX_1)$. A similar argument holds if $(\psi_\eta)_2(\bar x)= m(\bar x)$, so that in any case, $\bar \eta\geq (h_+-h)/(C\max(X_1,X_2))$, and $\psi_{\bar\eta}(x_1+\cdot)\leq m$, $\psi_{\bar\eta}(x_1+\cdot)\leq m$ on $\left[x_1,x_1+2\pi/\sqrt{4h-c^2}\right]$, as soon as $x_1\geq x_0$. In particular, for any $x_1\geq x_0$,
\[w\left(x_1+\pi/\sqrt{4h-c^2}\right)\geq \frac{h_+-h}{C\max(X_1,X_2)}e^{-\frac {c\pi}{2\sqrt{4h-c^2}}}X_1>0,\]
which is a contradiction, since $w(x)\to_{x\to\infty}0$ thanks to Proposition~\ref{prop:propright}.
\end{proof}

\section{Proof of Theorem \ref{thm:monotonicity}}
\label{sec:monotonicity}

\subsection{General case}
\label{sub:general_case}

%
%
%
%
%

%


The proof of the next lemma is based on a phase-plane-type analysis, see Figure~\ref{fig-phase}

\begin{figure}[h]
\centering
\includegraphics{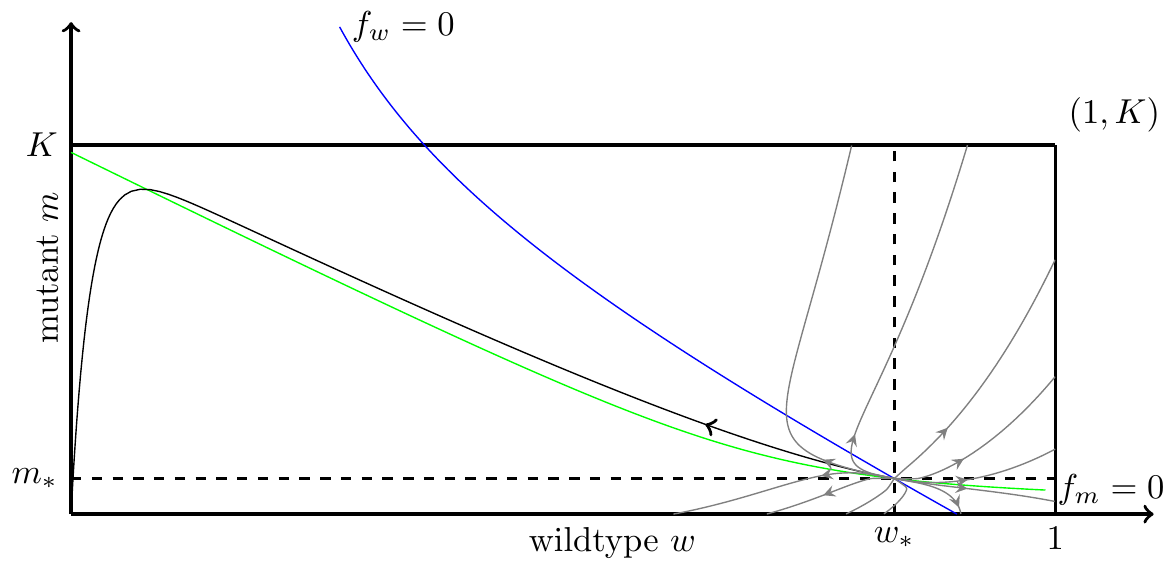}
\caption{Phase-plane-type representation of a solution of \eqref{eq:pbbox}: we represent (dark line) $x\mapsto (w(x),m(x))\in [0,1]\times[0,K]$. Note that the usual phase-plane for \eqref{eq:pbbox} is of dimension $4$. The  blue line represents the set of $(w,m)$ such that $f_w(w,m)=0$ (see \eqref{def_f}), $f_w(w,m)>0$ for $(m,w)$ on the left of the blue curve. The green line represent the set of $(w,m)$ such that $f_m(w,m)=0$ (see \eqref{def_f}), $f_m(w,m)>0$ for $(m,w)$ under the green curve. the gray lines represent several other solutions of \eqref{systemefront} such $(w(-a),m(-a))=(w^*,m^*)$. The dashed dark lines separate this phase plane into four compartiments that will be used in the third step of the proof of Lemma~\ref{lem:plan_phase}. \qg{Finally, the solid black line corresponds to the travelling wave.}}
\label{fig-phase} 
\end{figure}

\begin{lem}\label{lem:plan_phase}
Let $r,\,K,\,\mu$ satisfy Assumption~\ref{ass}. Let $(c,w,m)\in \mathbb R_+\times C^0([-a,a])^2 $ be a solution of \eqref{eq:pbbox}. 
Then there exists $\bar x\in[-a,0)$ such that one of the following properties is satisfied:
\begin{itemize}
 \item $w$ is decreasing on $[-a,a]$, while $m$ is increasing on $[-a,\bar x]$ and decreasing on $[\bar x,a]$,
\item $m$ is decreasing on $[-a,a]$, while $w$ is increasing on $[-a,\bar x]$ and decreasing on $[\bar x,a]$.
\end{itemize}
\end{lem}

\begin{proof}[Proof of Lemma \ref{lem:plan_phase}]

\noindent\emph{Step 1: sign of $f_w$ and $f_m$.} 

We recall the definition \eqref{def_f} of $f_w,\,f_m$. 
The inequality $f_w(w,m)\geq 0$ is equivalent, for $w\in[0,1]$ and $m\in[0,K]$, to
\begin{equation}\label{def:phiw}
 w\leq \phi_w(m):= \frac 12\left[1-\mu-m+\sqrt{(1-\mu-m)^2+4\mu m}\right].
\end{equation}
Notice that $m\in[0,K]\mapsto \phi_w(m)$ is a decreasing function (see Lemma~\ref{lem:assumption2}), that divides the square $\{(w,m)\in[0,1]\times[0,K]\}$ into two parts. 

Similarly, $f_m(w,m)\geq 0$ is equivalent, for $w\in[0,1]$ and $m\in[0,K]$, to
\begin{equation}\label{def:phim}
 m\leq \phi_m(w):= \frac 12\left[K-\frac{\mu K}{r}-w+\sqrt{\left(K-\frac{\mu K}{r}-w\right)^2+4\frac{\mu K}{r}w}\right].
\end{equation}
Here also, $w\in[0,1]\mapsto \phi_m(w)$ is a decreasing function (see Lemma~\ref{lem:assumption2}), since $\mu\leq 1/2$ (see Assumption \ref{ass}), that divides the square $\{(w,m)\in[0,1]\times[0,K]\}$ into two parts. 



\medskip

\noindent\emph{Step 2: possible monotony changes of $w(x),\,m(x)$.} 
Let $(c,w,m)\in \mathbb R_+\times C^0([-a,a])^2$ be a solution of \eqref{eq:pbbox}. If $w'(x)\geq 0$ for some $x> -a$, we can define $\bar x:=\inf\{y\geq x;\,w'(y)< 0\}$. Then $w'(\bar x)=0$, and $w''(\bar x)\leq 0$, which implies
$$f_w(w(\bar x),m(\bar x))=-cw'(\bar x)-w''(\bar x) \geq 0,$$
that is $w(\bar x)\leq \phi_w(m(\bar x))$. The symmetric property also holds: if $w'(x)\leq 0$ for some $x> -a$, we can define $\bar x:=\inf\{y\geq x;\,w'(y)> 0\}$, and then, $w(\bar x)\geq \phi_w(m(\bar x))$.

We repeat the argument for the function $m$: let $(c,w,m)\in \mathbb R_+\times C^0([-a,a])^2$ be a solution of \eqref{eq:pbbox}. If $m'(x)\geq0$ for some $x> -a$, we can define $\bar x:=\inf\{y\geq x;\,m'(y)< 0\}$, and then, $m(\bar x)\leq \phi_m(w(\bar x))$. Finally, if $m'(x)\leq 0$ for some $x> -a$, we can define $\bar x:=\inf\{y\geq x;\,m'(y)> 0\}$, and then, $m(\bar x)\geq \phi_m(w(\bar x))$.

\medskip

\noindent\emph{Step 3: phase plane analysis}
Notice that $(w(-a),m(-a))=(w^*,m^*)$, and then,
\begin{equation}\label{phi_bar_mw}
m(-a)=\phi_m(w(-a)),\quad w(-a)=\phi_w(m(-a)).
\end{equation}
We will consider now consider individually the four possible signs of $w'(-a), $ $ m'(-a)$ (the cases where $w'(-a)=0$ or $m'(-a)=0$ will be considered further in the proof):
%
%
%

(i) Assume that $w'(-a)>0$ and $m'(-a)>0$. We define $\bar x:=\inf\{y\geq {\color{red}-a};\,w'(y)< 0\textrm{ or }m'(y)<0\}$. Since $w$ and $m$ are increasing on $[-a,\bar x]$, \eqref{phi_bar_mw} holds and $w\mapsto \phi_m(w)$, $m\mapsto \phi_w(m)$ are decreasing functions, we have 
$$w(\bar x)>w(-a)=\phi_w(m(-a))\geq\phi_w(m(\bar x)),$$
$$m(\bar x)>m(-a)=\phi_m(w(-a))\geq\phi_m(w(\bar x)).$$
Then, $f_w(w(\bar x),m(\bar x))<0$ and $f_m(w(\bar x),m(\bar x))<0$. It then follows from Step~2 that $\bar x=a$, which means that $w$ and $m$ are increasing on $[-a,a]$. It is a contradiction, since $0=w(a)<w(-a)=\bar w$.

Notice that the same argument would also work \qg{on $[x,a]$, for any $ (w(x), m(x)) $ that satisfies $ w(x)>\Phi_w(m(x))$, $m(x)> \Phi_m(w(x))$, $w'(x)>0$ and $m'(x)>0$.} 

(ii) Assume that $w'(-a)<0$ and $m'(-a)<0$. Let $\bar x:=\inf\{y\geq -a;\,w'(y)> 0\textrm{ or }m'(y)>0\}$. Since $w$ and $m$ are decreasing on $[-a,\bar x]$, \eqref{phi_bar_mw} holds and $w\mapsto \phi_m(w)$, $m\mapsto \phi_w(m)$ are decreasing functions, we have
$$w(\bar x)<w(-a)=\phi_w(m(-a))\leq\phi_w(m(\bar x)),$$
$$m(\bar x)<m(-a)=\phi_m(w(-a))\leq \phi_m(w(\bar x)).$$
It then follows from Step~2 that $\bar x=a$, which means that $w$ and $m$ are non-increasing on $[-a,a]$. Notice that this is not a contradiction, since $w(a)=0<w^*=w(-a)$, $m(a)=0<m^*=m(-a)$.

Notice that the same argument would work \qg{on $[x,a]$, for any $ (w(x), m(x)) $ that satisfies $ \Phi_w(m(x))> w(x) $, $ m(x)<\Phi_m(w(x))$, $w'(x)<0$ and $m'(x)<0$. }

(iii) Assume that $w'(-a)<0$ and $m'(-a)>0$. We define $\bar x:=\inf\{y\geq -a;\,w'(y)> 0\textrm{ or }m'(y)<0\}$. The argument used in the two previous cases cannot be employed here.
We know however that  $w(\bar x)<w^*$, $m(\bar x)>m^*$. Since $m(a)=0<m^*$, it implies in particular that $\bar x<a$, and, with the notations of Lemma~\ref{lem:zerosinclusions}, $(w(\bar x),m(\bar x))\in \mathcal D_l$.


If $w'$ changes sign in $\bar x$, then Step 2 implies that $w(\bar x)\geq \phi_w(m(\bar x))$, that is, with the notations of Lemma~\ref{lem:zerosinclusions}, $(w(\bar x),m(\bar x))\in Z_w^-$. Thanks to Lemma~\ref{lem:zerosinclusions}, it follows that $(w(\bar x),m(\bar x))\in Z_w^-\cap \mathcal D_l\subset Z_m^-$, and then $m(\bar x)>\phi_m(w(\bar x))$, which implies $-cm'(\bar x)-m''(\bar x)=f_m(w(\bar x),m(\bar x))<0$. If $m'(\bar x)=0$, then $m''(\bar x)>0$, which is incompatible with the fact that  $m'\geq 0$ on $[-a,\bar x)$ and $m'(\bar x)=0$. We have thus shown that $m'(\bar x)>0$. Thanks to the definition of $ \bar x$, 
either $w$ is locally increasing near $ \bar x^+, $ or there exists a sequence $ (x_n)\to \bar x^+ $ such that $ w'(x_{2n})>0 $ and $ w'(x_{2n+1})<0. $ In the first case, for $ \varepsilon>0 $ small enough, $ w(\bar x+\varepsilon)>w(\bar x)\geq \Phi_w(m(\bar x))\geq\Phi_w(m(\bar x+\varepsilon))$ along with $ w'(\bar x+\varepsilon)>0. $ In the second case, $ w''(\bar x)=0 $, then $f_w(w, m)(\bar x)=0$, and a simple computation shows that for $ \varepsilon>0 $ small enough, 
\[ f_w\left(w(\bar x+\varepsilon), m(\bar x+\varepsilon)\right)=(\mu-w(\bar x))\varepsilon m'(\bar x)+o(\varepsilon)<0, \]
where we have used the fact that $\mu-w(\bar x)<0$ (since $w(\bar x)\geq \phi_w(m(\bar x))\subset \phi_w([0,K])\subset(\mu,\infty)$, see Lemma~\ref{lem:assumption2}). In any case, for some $\varepsilon>0$ arbitrarily small, $w(\bar x+\varepsilon)> \phi_w(m(\bar x+\varepsilon))$, $m(\bar x+\varepsilon)>\phi_m(w(\bar x+\varepsilon))$, along with $w'(\bar x+\varepsilon)>0$ and $m'(\bar x+\varepsilon)>0$. argument (i) can now be applied to $(w,m)|_{[\bar x+\varepsilon,a]}$, leading to a contradiction.

If $m'$ changes sign in $\bar x$, then Step 2 implies that $m(\bar x)\leq \phi_m(w(\bar x))$, that is, with the notations of Remark~\ref{Rk:zerosinclusions}, $(w(\bar x),m(\bar x))\in Z_m^+$. Thanks to Remark~\ref{Rk:zerosinclusions}, it follows that $(w(\bar x),m(\bar x))\in Z_m^+\cap \mathcal D_l\subset Z_w^+$, and then $w(\bar x)<\phi_w(m(\bar x))$,  which implies $-cw'(\bar x)-w''(\bar x)=f_w(w(\bar x),m(\bar x))>0$. If $w'(\bar x)=0$, then $w''(\bar x)<0$, which is incompatible with the fact that  $w'\leq 0$ on $[-a,\bar x)$ and $w'(\bar x)=0$. We have thus shown that $w'(\bar x)<0$. Thanks to the definition of $ \bar x$, 
either $m$ is locally decreasing near $ \bar x^+, $ or there exists a sequence $ (x_n)\to \bar x^+ $ such that $ m'(x_{2n})>0 $ and $ m'(x_{2n+1})<0. $ In the first case, for $ \varepsilon>0 $ small enough, $ m(\bar x+\varepsilon)<m(\bar x)\leq \Phi_m(w(\bar x))\leq\Phi_m(w(\bar x+\varepsilon))$ along with $ m'(\bar x+\varepsilon)>0. $ In the second case,  $ m''(\bar x)=0 $, then $f_m(w, m)(\bar x)=0$, and a simple computation shows that for $ \varepsilon>0 $ small enough, 
\[ f_m\left(w(\bar x+\varepsilon), m(\bar x+\varepsilon)\right)=\left(\mu-\frac{r}{K}m(\bar x)\right)\varepsilon w'(\bar x))+o(\varepsilon)>0, \]
where we have used the fact that $\mu-\frac{r}{K}m(\bar x)<0$  (since $m(\bar x)>m^*=\phi_m(w^*)\subset \phi_m([0,1])\subset (\mu K/r,\infty)$, see Lemma~\ref{lem:assumption2}). In both cases, argument~(ii) can now be applied to $(w,m)|_{[\bar x+\varepsilon,a]}$, which is not a contradiction, since $w(a)=0<w(\bar x)$, $m(a)=0<m(\bar x)$.

(iv) Assume that $w'(-a)>0$ and $m'(-a)<0$. We define $\bar x:=\inf\{y\geq -a;\,w'(y)< 0\textrm{ or }m'(y)>0\}$. Then $w(\bar x)>w^*$, $m(\bar x)<m^*$. Since $m(a)=0<m^*$, it implies in particular that $\bar x<a$, and, with the notations of Lemma~\ref{lem:zerosinclusions}, $(w(\bar x),m(\bar x))\in \mathcal D_r$.

If $w'$ changes sign in $\bar x$, then Step 2 implies that $w(\bar x)\leq \phi_w(m(\bar x))$, that is, with the notations of Remark~\ref{Rk:zerosinclusions}, $(w(\bar x),m(\bar x))\in Z_w^+$. Thanks to Remark~\ref{Rk:zerosinclusions}, it follows that $(w(\bar x),m(\bar x))\in Z_w^+\cap \mathcal D_r\subset Z_m^+$, and then $m(\bar x)<\phi_m(w(\bar x))$, which implies $-cm'(\bar x)-m''(\bar x)=f_m(w(\bar x),m(\bar x))>0$. If $m'(\bar x)=0$, then $m''(\bar x)<0$, which is incompatible with the fact that  $m'\leq 0$ on $[-a,\bar x)$ and $m'(\bar x)=0$. We have thus shown that $m'(\bar x)<0$. Thanks to the definition of $ \bar x$, 
either $w$ is locally decreasing near $ \bar x^+, $ or there exists a sequence $ (x_n)\to \bar x^+ $ such that $ w'(x_{2n})>0 $ and $ w'(x_{2n+1})<0. $ In the first case, for $ \varepsilon>0 $ small enough, $ w(\bar x+\varepsilon)<w(\bar x)\leq \Phi_w(m(\bar x))\leq\Phi_w(m(\bar x+\varepsilon))$ along with $ w'(\bar x+\varepsilon)<0. $ In the second case, $ w''(\bar x)=0 $, then $f_w(w, m)(\bar x)=0$, and a simple computation shows that for $ \varepsilon>0 $ small enough, 
\[ f_w(w, m)(\bar x+\varepsilon)=(\mu-w(\bar x))\varepsilon m'(\bar x)+o(\varepsilon)>0, \]
where we have used the fact that $\mu-w(\bar x)<0$  (since $w(\bar x)>w^*=\phi_w(m^*)\subset \phi_w([0,1])\subset (\mu,\infty)$, see Lemma~\ref{lem:assumption2}). In both cases, argument~(ii) can now be applied to $(w,m)|_{[\bar x+\varepsilon,a]}$, which is not a contradiction, since $w(a)=0<w(\bar x)$, $m(a)=0<m(\bar x)$.

If $m'$ changes sign in $\bar x$, then Step 2 implies that $m(\bar x)\geq \phi_m(w(\bar x))$, that is, with the notations of Lemma~\ref{lem:zerosinclusions}, $(w(\bar x),m(\bar x))\in Z_m^-$. Thanks to Lemma~\ref{lem:zerosinclusions}, it follows that $(w(\bar x),m(\bar x))\in Z_m^-\cap \mathcal D_r\subset Z_w^-$, and then $w(\bar x)>\phi_w(m(\bar x))$, which implies $-cw'(\bar x)-w''(\bar x)=f_w(w(\bar x),m(\bar x))<0$. If $w'(\bar x)=0$, then $w''(\bar x)>0$, which is incompatible with the fact that  $w'\geq 0$ on $[-a,\bar x)$ and $w'(\bar x)=0$. We have thus shown that $w'(\bar x)>0$. Thanks to the definition of $\bar x$, 
either $m$ is locally increasing near $ \bar x^+, $ or there exists a sequence $ (x_n)\to \bar x^+ $ such that $ m'(x_{2n})>0 $ and $ m'(x_{2n+1})<0. $ In the first case, for $ \varepsilon>0 $ small enough, $ m(\bar x+\varepsilon)>m(\bar x)\geq \Phi_m(w(\bar x))\geq\Phi_m(w(\bar x+\varepsilon))$ along with $ m'(\bar x+\varepsilon)>0. $ In the second case,  $ m''(\bar x)=0 $, then $f_m(w, m)(\bar x)=0$, and a simple computation shows that for $ \varepsilon>0 $ small enough, 
\[ f_m(w, m)(\bar x+\varepsilon)=\left(\mu-\frac{r}{K}m(\bar x)\right)\varepsilon w'(\bar x)+o(\varepsilon)<0,\]
where we have used the fact that $\mu-\frac{r}{K}m(\bar x)<0$  (since $m(\bar x)\geq \phi_m(w(\bar x))\subset \phi_m([0,1])\subset(\mu K/r,\infty)$, see Lemma~\ref{lem:assumption2}). In both cases, argument~(i) can now be applied to $(w,m)|_{[\bar x+\varepsilon,a]}$, leading to a contradiction.

\medskip

Let consider now the case where $w'(-a)=0$ or $m'(-a)=0$. If $w'(-a)=m'(-a)=0$, then $w\equiv w^*$, $m\equiv m^*$, which is a contradiction. Assume w.l.o.g. that $w'(-a)\neq 0$. If there exists $\varepsilon>0$ such that for any $x\in [-a,-a+\varepsilon]$, $m'(x)=0$, then $m$ is constant on the interval $[-a,-a+\varepsilon]$, and then $f_m(w(x),m(x))=0$ for $x\in [-a,-a+\varepsilon]$. This implies in turn that $m(x)=\phi_m(w(x))$, and then $w$ is constant on $[-a,-a+\varepsilon]$, since $\phi_m$ is a decreasing function, which is a contradiction. There exists thus a sequence $x_n\to -a$, $x_n>-a$, such that $w(x_n)\neq w^*$, and $\textrm{sgn}(m(x_n)-m^*)=\textrm{sgn}(m'(x_n))\neq 0$, while $\textrm{sgn}(w(x_n)-w^*)=\textrm{sgn}(w'(0))\neq 0$. The above argument (i-iv) can therefore be reproduced for $(w,m)|_{[x_n,a]}$. 

Finally, the fact that $\bar x\leq 0$ is a consequence of $w(0)+m(0)<\min (w^*,m^*)$.
\end{proof}

\begin{prop}\label{prop:monotonicity}
Let $r,\,K,\,\mu$ satisfy Assumption~\ref{ass}. Let $(c,w,m)\in  \mathbb R_+\times C^0(\mathbb R)^2$ be a solution of \eqref{systemefront} constructed in Theorem \ref{thm:main}. Then, there exists $\bar x\in[-\infty,0)$ such that
\begin{itemize}
 \item either $w$ is decreasing on $\mathbb R$, while $m$ is increasing on $(-\infty,\bar x]$ and decreasing on $[\bar x,\infty)$,
\item or $m$ is decreasing on $\mathbb R$, while $w$ is increasing on $(-\infty,\bar x]$ and decreasing on $[\bar x,\infty)$,
\end{itemize}
Moreover, 
$$w(x)\to w^*,\quad m(x)\to m^*\textrm{ as }x\to-\infty.$$
\end{prop}

\begin{proof}[Proof of Proposition \ref{prop:monotonicity}]
The travelling wave $(c,w,m)$ constructed in Theorem \ref{thm:main} is obtained as a limit (in $L^\infty_{loc}(\mathbb R)$) of solutions $(w_n,m_n,c_n)\in \mathbb R_+\times C^0([-a_n,a_n])^2$ of \eqref{eq:pbbox} on $[-a_n,a_n]$, with $a_n\underset{n\to\infty}{\longrightarrow}\infty$. Each of those solutions then satisfy one of the two the monotonicity properties of Lemma~\ref{lem:plan_phase}. In particular, there is at least one of those properties that is satisfied by an infinite sequence of solutions $(w_n,m_n,c_n)$. We may then assume w.l.o.g. that all the solutions $(w_n,m_n,c_n)$ satisfy the first monotonicity property in Lemma~\ref{lem:plan_phase}. We assume therefore that for all $n\in\mathbb N$, there exists $\bar x_n\in[-a_n,0)$ such that $w_n$ is decreasing on $[-a_n,a_n]$, while $m_n$ is increasing on $[-a_n,\bar x_n]$ and decreasing on $[\bar x_n,a_n]$. Up to an extraction, we can define $\bar x:=\lim_{n\to\infty}a_n\in [-\infty,0]$. Then, $w$ is a uniform limit of decreasing function on any bounded interval, and is thus decreasing. Let now $\tilde x>\bar x$.  $m_n$ is then a decreasing function on $[\tilde x,\infty)\cap[-a_n,a_n]$ for $n$ large enough, and $m|_{[\tilde x,\infty)}$ is thus a uniform limit of decreasing functions on any bouded interval of $[\tilde x,\infty)$. This implies that $m$ is decreasing on $[\bar x,\infty)$. A similar argument shows that $m$ is increasing on $(-\infty,\bar x]$, if $\bar x>-\infty$. the case where all the solutions $(w_n,m_n,c_n)$ satisfy the second monotonicity property in Lemma~\ref{lem:plan_phase} can be treated similarly.


\medskip

We have shown in particular that $w$, $m$ are monotonic on $(-\infty,\tilde x)$, for some $\tilde x<0$ ($\tilde x=\bar x$ if $\bar x>-\infty$, $\tilde x=0$ otherwise). Since $w$ and $m$ are regular bounded functions, it implies that 
\[ f_w(w(x),m(x))=-cw'(x)-w''(x)\to 0, \]
\[ f_m(w(x),m(x))=-cm'(x)-m''(x)\to 0, \]
 as $ x\to -\infty.$
This combined to $\liminf_{x\to-\infty}w(x)+m(x)>0$  and $(w,m)\in[0,1]\times [0,K]$ implies that $w(x)\to w^*$ and $m(x)\to m^*$ as $x\to-\infty.$
\end{proof}

\subsection{Case of a small mutation rate}
\label{sub:small_mu}

The result of this subsection shows that if $\mu>0$ is small, then only the first situation described in Lemma \ref{lem:plan_phase}, with $\bar x>-\infty$, is possible.

\begin{prop}\label{prop:monotonicity_mu_petit}
Let $r,\,K,\,\mu$ satisfy Assumption~\ref{ass}. Let $(c,w,m)\in \mathbb R_+\times C^0(\mathbb R)$ be a solution of \eqref{systemefront} constructed in Theorem \ref{thm:main}. There exists $\bar \mu=\bar \mu(r,K)>0$ such that $\mu<\bar \mu$ implies that $w$ is decreasing on $\mathbb R$, while  $m$ is increasing on $(-\infty,\bar x]$ and decreasing on $[\bar x,\infty)$, for some $\bar x\in \mathbb R_-$.
\end{prop}

\begin{proof}[Proof of Proposition \ref{prop:monotonicity_mu_petit}]
Notice that the solution $(c,w,m)$ satisfies the assumptions of Proposition \ref{prop:monotonicity}. 

Let us assume that $\|m\|_\infty\leq m^*$. We will show that this assumption leads to a contradiction if $\mu>0$ is small. Let $\bar x=\max\left\{x>-\infty;\, w(x)\geq  m^*\right\}$. Then $w$ satisfies $-cw'-w''\leq (1-\mu)w+\mu\,m^*$ on $(-\infty,\bar x]$. Since $(c,w,m)$ satisfies the assumptions of Proposition \ref{prop:monotonicity} and $w(\bar x)=m^*<w^*$, we have that $w(x)\geq  m^*$ for all $x\leq \bar x$. $w$ thus satisfies $-cw'-w''\leq w$ on $(-\infty,\bar x]$. We define now
$$\bar w(x)=m^*\, e^{\frac{c-\sqrt{c^2-4}}2(\bar x-x)},$$
which satisfies $-c\bar w'-\bar w''=\bar w$ on $(-\infty,\bar x]$, $\bar w(\bar x)=w(\bar x)$, and $\bar w(y)\geq 1\geq w(y)$ for $y<<0$. Since $w$ is bounded, $\alpha \bar w>w$ for $\alpha>0$ large enough. We can then define $\alpha^*:=\min\{\alpha>0;\, \alpha \bar w>w \textrm{ on }(-\infty,\bar x)\}$. If $\alpha^*>1$, there exists $x^*\in (-\infty,\bar x)$ such that $\alpha^* \bar w(x^*)= w(x^*)$, and then, $-c(\alpha^*\bar w- w)'(x^*)-(\alpha^*\bar w- w)''(x^*)>\alpha^*\bar w(x^*)- w(x^*)=0$, which is a contradiction, since $\alpha^* \bar w> w$ implies that $-c(\alpha^*\bar w- w)'(x^*)-(\alpha^*\bar w- w)''(x^*)\leq 0$. Thus,
\begin{equation}\label{borne_sup_w}
w(x)\leq\bar w(x)
,\textrm{ for }x\in(-\infty,\bar x].
\end{equation}
In particular, if we define 
\begin{equation}\label{def:tildex}
 \tilde x:=\bar x-\frac 2{c-\sqrt{c^2-4}}\ln\qg{\left(\frac K{m^*}\left(\frac 12-\frac\mu r-\frac 1{2r}\right)-1\right)},
\end{equation}
then \qg{ $ w(x)\leq K\left(\frac 12-\frac \mu r-\frac 1{2r}\right)-m^*$ } on $[\tilde x,\bar x]$. \qg{Notice that $m^*\to 0$ as $\mu\to 0$ (see Lemma~\ref{lem:estmu0}), and then $\frac K{m^*}\left(\frac 12-\frac \mu r-\frac 1{2r}\right)\to \infty$ as $\mu\to 0$; $\tilde x$ is then well defined as soon as $ \mu>0$ is small enough, and $\tilde x-\bar x\to -\infty$ as $\mu\to 0$.} This estimate applied to the equation on $m$ (see \eqref{systemefront}), implies, for $x\in[\tilde x,\bar x]$, that
$$-cm'(x)-m''(x)\geq r\left(1-\frac \mu r-\frac{m^*+w(x)}K\right)m(x)+\mu\,w\geq \frac {1+r}2m+\mu\,m^* ,\textrm{ for }x\in(\tilde x,\bar x],$$
where we have also used the fact that $w\geq m^*$ on $(-\infty,\bar x]$.

\medskip

We define next
$$\bar m_1:=-\frac{\mu\,m^*}{c+2}(x-\bar x)\left(x-(\bar x-1)\right),$$
which satisfies $-c\bar m_1'-\bar m_1''\leq \mu m^*$ as well as $\bar m_1\left(\bar x-1\right)=0\leq m\left(\bar x-1\right)$ and $\bar m_1(\bar x)=0\leq m(\bar x)$. The weak maximum principle (\cite{GT98}, Theorem 8.1) then implies that $ m(x)\geq \bar m_1(x)$ for all \qg{$x\in\left[\bar x-1, \bar x\right]$,} and in particular,
$$m(\bar x-1/2)\geq \qg{ \frac{\mu\,m^*}{4(c+2)}.}$$
We define (we recall the definition \eqref{def:tildex} of $\tilde x$)
$$\bar m_2:= \frac{\mu\, m^*}{4(c+2)}e^{\frac{c-\sqrt{c^2-2(1+r)}}2\left(\bar x-1/2-x\right)}-Ae^{\frac{c}2\left(\bar x-1/2-x\right)},$$
with $A= \frac{\mu\, m^*}{4(c+2)}e^{-\frac{\sqrt{c^2-2(r+1)}}2(\bar x-1/2-\tilde x)}$, so that $\bar m_2(\tilde x)=0$. $\bar m_2$ then satisfies $-c\bar m_2'-\bar m_2''< \frac {(1+r)}2\bar m_2$, since $c(c/2)-(c/2)^2> \frac{1+r}2$ (see \eqref{eq:defminc}).  Let now 
$$\alpha^*:=\max\{\alpha;m(x)\geq \alpha \bar m_2(x),\,\forall x\in [\tilde x,\bar x-1/2]\}.$$
$\alpha^*>0$, since $\min_{[\tilde x,\bar x-1/2]}m>0$. If $\alpha^*<1$, then $\alpha^*\bar m_2(\bar x-1/2)< \frac{\mu\,m^*}{4(c+2)}\leq m(\bar x-1/2)$, while $\alpha^*\bar m_2(\tilde x)=0<m(\tilde x)$. Then $\alpha^*\bar m_2\leq m$ on $[\tilde x,\bar x-1/2]$, and there exists $x^*\in[\tilde x,\bar x-1/2]$ such that $\alpha^*\bar m_2(x^*)=m(x^*)$, and

$$0\leq -c(\bar m_2-m)'(x^*)-(\bar m_2-m)''(x^*)<\frac {(1+r)}2(\bar m_2-m)(x^*)=0,$$
which is a contradiction. We have thus proven that $m\geq \bar m_2$ on $[\tilde x,\bar x-1/2]$, and in particular, for $\mu>0$ small enough, 
$$\|m\|_\infty\geq \bar m_2\left(\tilde x+\frac{2\ln(2)}{\sqrt{c^2-2(1+r)}}\right)=\frac{\mu\,m^*}{4(c+2)}e^{\frac{c\ln 2}{\sqrt{c^2-2(1+r)}}}e^{\frac{c-\sqrt{c^2-2(1+r)}}2\left(\bar x-1/2-\tilde x\right)}.$$
We recall indeed that $\tilde x-\bar x\to -\infty$ as $\mu\to 0$, and then $\tilde x+\frac{2\ln(2)}{\sqrt{c^2-2(1+r)}}\in [\tilde x,\bar x-1/2]$ if $\mu>0$ is small enough. Thanks to the definition of $\tilde x$, this inequality can be written
\begin{align*}
&\ln\left(\frac{4(c+2) \|m\|_\infty}{\mu\,m^*}\right)-\frac{c\ln 2}{\sqrt{c^2-2(1+r)}}\\
&\quad \geq \frac{c-\sqrt{c^2-2(1+r)}}2\left(-1/2+\frac 2{c-\sqrt{c^2-4}}\qg{\ln\left(\frac K{ m^*}\left(\frac 12-\frac \mu r-\frac 1{2r}\right)-1\right)}\right).
\end{align*}

We have assumed that $\|m\|_\infty=m^*$, thus, if  we denote by $\mathcal O_{\mu\sim 0^+}(1)$ a function of $\mu>0$ that is bounded for $\mu$ small enough, we get
$$\ln\left(\frac 1\mu\right)+\mathcal O_{\mu\sim 0^+}(1)\geq \frac{c-\sqrt{c^2-2(1+r)}}{c-\sqrt{c^2-4}}\ln\left(\frac 1{ m^*}\right).$$
Moreover, we know that $m^*\leq C \mu$ for some $C> 0$, see Lemma \ref{lem:estmu0}. Then, 
$$\ln\left(\frac 1\mu\right)+\mathcal O_{\mu\sim 0^+}(1)\geq \frac{c-\sqrt{c^2-2(1+r)}}{c-\sqrt{c^2-4}}\ln\left(\frac 1{\mu}\right),$$
which is a contradiction as soon as $\mu>0$ is small, since $r>1$.

\medskip

We have thus proved that for $\mu>0$ small enough, we have $\|m\|_\infty> m^*$. This estimate combined to Proposition~\ref{prop:monotonicity} proves Proposition~\ref{prop:monotonicity_mu_petit}.
\end{proof}

\section{Proof of Theorem \ref{thm:KPP}}
\label{sec:K_small}

Notice first that if we chose $\varepsilon>0$ small enough, then $0<\mu<K<\varepsilon$ implies that Assumption~\ref{ass} is satisfied.

We will need the following estimate on the behavior of travelling waves of \eqref{systemefront}:
\begin{prop}\label{prop:w_left}
Let $r,\,K,\,\mu$ satisfy Assumption~\ref{ass}. Let $(c,w,m)\in \qg{\mathbb R_+\times C^\infty(\mathbb R)\times C^\infty(\mathbb R)}$ be a solution of \eqref{systemefront}, such that $\liminf_{x\to -\infty}(w(x)+m(x))>0$. Then, $\liminf_{x\to -\infty} w(x)\geq 1-\mu-K$.

Moreover, if $ w(\bar x)<1-K-\mu$ for some $\bar x\in\mathbb R$, then $w$ is decreasing on $[\bar x,\infty)$.
\end{prop}
\begin{proof}[Proof of Proposition \ref{prop:w_left}]
 Since $m(x)< K$ for all $x\in\mathbb R$, any local minimum $\bar x$ of $w$ satisfies
\begin{eqnarray}
 0&\geq& -cw'(\bar x)-w''(\bar x)=(1-\mu-w(\bar x)-m(\bar x))w(\bar x)+\mu m(\bar x)\nonumber\\
&>&(1-\mu-K-w(\bar x))w(\bar x),\label{control_w}
\end{eqnarray}
and then $w(\bar x)\geq 1-\mu-K$. 

Assume that $\liminf_{x\to -\infty}w(x)<1-\mu-K$. Then, $x\mapsto w(x)$ can not have a minimum for $x<<0$, and is thus  monotonic for $x<<0$. Then $l:=\lim_{x\to-\infty}w(x)\in [0,1-\mu-K]$ exists and $w'(x)\to_{x\to -\infty}0$, $w''(x)\to_{x\to -\infty}0$. This implies $-cw'(x)-w''(x)\to_{x\to\infty} 0$, which, coupled to \eqref{control_w} implies that $l=0$ or $l=1-\mu-K$. $l=0$ leads to a contradiction, since $\liminf_{x\to -\infty}(w(x)+m(x))>0$, which proves the first assertion.

\medskip

To prove the second assertion, we notice that since $w$ cannot have a minimum $\tilde x\in\mathbb R$ such that $w(\tilde x)<1-K-\mu$, $w$ is monotonic on $\{x\in\mathbb R;\;w(x)<1-K-\mu\}$. This monotony combined to $\liminf_{x\to -\infty}w(x)\geq1-\mu-K>w(\bar x)$ implies that $w$ is decreasing on $[\bar x,\infty)$.
\end{proof}

The main idea of the proof of theorem~\ref{thm:KPP} is to compare $w$ to solutions of modified Fisher-KPP equations, which we introduce in the following lemma:
\begin{lem}
Let $r,\,K,\,\mu$ satisfy Assumptions~\ref{ass}. Let $(c,w,m)\in\qg{\mathbb R_+\times C^\infty(\mathbb R)\times C^\infty(\mathbb R)}$ be a solution of \eqref{systemefront}, with $c\geq 2+K$. Let also $\overline w\in C^\infty(\mathbb R)$, $\underline w\in C^\infty(\mathbb R)$ solutions of
\begin{equation}\label{overline_w}\left\{\begin{array}{l}
   -c\overline w'-\overline w''=\overline w(1-\overline w)+K,\\
\overline w(x)\to_{x\to-\infty} \frac{1+\sqrt{1+4K}}{2},\; \overline w(x)\to_{x\to+\infty}-\frac{\sqrt{1+4K}-1}{2},
  \end{array}\right.
\end{equation}

\begin{equation}\label{underline_w}\left\{\begin{array}{l}
   -c\underline w'-\underline w''=\underline w(1-2K-\underline w),\\
\underline w(x)\to_{x\to-\infty}1-2K,\; \underline w(x)\to_{x\to+\infty}0.\phantom{sfqsqdgdfg}
  \end{array}\right.
\end{equation}
%
%
  Assume $ \underline w(0)\leqslant w(0)\leqslant \overline w(0)$. Then 
 \[ \forall x\leqslant 0, \qquad \underline w(x)\leqslant w(x)\leqslant \overline w(x). \]
 \label{thm:slidingsubsupKPP}
\end{lem}
\begin{Rk}\label{rem:w_KPP}
Notice that $\overline w+\frac{\sqrt{1+4K}-1}2$ and $\underline w$ are solution of a classical Fisher-KPP equation $-cu'-u''=u(a-bu)$ with $a\in(0,1+2K)$, $b>0$, and a speed $c\geq 2\sqrt a$. The existence, uniqueness (up to a translation) and monotony of $\overline w$ and $\underline w$ are thus classical results (see e.g. \cite{KPP1937}). Thanks to those relations, the argument developed in this section can indeed be seen as a precise analysis on the profile of $x\mapsto u(x)$ for $x>0$ large.
\end{Rk}

\begin{proof}[Proof of Lemma \ref{thm:slidingsubsupKPP}] To prove this lemma, we use a sliding method.
 \begin{itemize}
  \item Let $\underline w_\eta(x):= \underline w(x+\eta)$. Thanks to Proposition~\ref{prop:w_left}, there exists $x_0\in \mathbb R$ such that $w(x)>1-2K=\sup_{\mathbb R}\underline w_\eta$ for all $x\leq x_0$ (we recall that $\mu<K$). Since $\lim_{x\to\infty}\underline w(x)=0$, there exists $x^0>0$ such that $\underline w(x)<\min_{[x_0,0]}w$ for all $x>x^0$. Then, for $\eta\geq x_0+x^0$,
  \[\underline w_\eta(x)<w(x),\quad \forall x\leqslant 0. \]
We can then define $ \overline \eta:=\inf\{\eta, \forall x\leqslant 0, \underline w_\eta(x)\leqslant w(x)\}$. We have then $\underline w_{\overline \eta}(x)\leqslant \qg{w(x)}$ for all $x\leq 0$. If $\overline\eta>0$, since $\inf_{(-\infty,x_0]}w>1-2K=\sup_{\mathbb R}\underline w_\eta$ and $\underline w_{\overline \eta}(0)=\overline w(\eta)<\overline w(0)\leq w(0)$ (we recall that $\underline w$ is decreasing, see Remark~\ref{rem:w_KPP}), there exists $ \underline x\in (x_0,0)$ such that $ \underline w_{\overline\eta}(\underline x)=w(\underline x)$. $ \underline x$ is then a minimum of $w-\underline w_{\overline \eta}$, and thus
\begin{eqnarray}
 0&\geq& -c(w-\underline w_{\overline \eta})'(\underline x)-(w-\underline w_{\overline \eta})''(\underline x)\nonumber\\
&=&w(\underline x)(1-\mu-m(\underline x)-w(\underline x))+\mu m(\underline x)-\underline w_{\overline \eta}(\underline x)(1-2K-\underline w_{\overline \eta}(\underline x))\nonumber\\
&>&w(\underline x)(1-2K-w(\underline x))-\underline w_{\overline \eta}(\underline x)(1-2K-\underline w_{\overline \eta}(\underline x))=0,\label{est1}
\end{eqnarray}
where we have used the estimate $\|m\|_\infty\leq K$ obtained in Proposition~\ref{thm:precisebound}. \eqref{est1} is a contradiction, we have then shown that $\bar \eta\leq 0$, and thus, for all $x\leq 0$, $\underline w(x)\leq w(x)$.
  
 \item Similarly, let $\overline w_\eta(x):=\overline w(x-\eta)$. Since $\lim_{x\to-\infty}\overline w(x)>1$ and $w$ satisfies the estimate of Proposition~\ref{thm:precisebound}, we have, for $\eta\in \mathbb R$ large enough,
  \[ \forall x\leqslant 0,\quad  w(x)<1<\overline w_\eta(x). \]
  We can then define $ \overline \eta:=\qg{\inf\{\eta, \forall x\leqslant 0,  w(x)\leqslant \overline w_\eta(x)\}}. $ We have then $w(x)\leqslant \overline w_{\overline \eta}$ for all $x\leq 0$. If $\overline\eta>0$, since $\sup_{\mathbb R}w<1< \lim_{x\to -\infty} \overline w(x)$ and $w(0)\leq \overline w(0)< \overline w(-\bar \eta)=\overline w_{\bar \eta}(0)$ (we recall that $\overline w$ is decreasing, see Remark~\ref{rem:w_KPP}), there exists $\bar x<0$ such that $w(\underline x)=\overline w_{\overline \eta}(\underline x)$. $\bar x$ is then a minimum of $\overline w_{\overline \eta}-w$, and thus
\begin{eqnarray*}
 0&\geq&-c(\overline w_{\overline \eta}-w)'(\bar x)-(\overline w_{\overline \eta}-w)''(\bar x)\\
&=&\overline w_{\overline \eta}(\bar x)(1-\overline w_{\overline \eta}(\bar x))+K-w(\bar x)(1-\mu-w(\bar x)-m(\bar x))-\mu m(\bar x)\\
&>&\overline w_{\overline \eta}(\bar x)(1-\overline w_{\overline \eta}(\bar x))-w(\bar x)(1-w(\bar x))=0,
\end{eqnarray*}
which is a contradiction. We have then shown that $\bar \eta\leq 0$, and thus, for all $x\leq 0$, $w(x)\leq \overline w(x)$.

 \end{itemize}

\end{proof}

We also need to compare the solution of the Fisher-KPP equation with speed $c$ to the solutions of the modified Fisher-KPP equations introduced in Lemma~\ref{thm:slidingsubsupKPP}.
\begin{lem}
Let $r,\,K,\,\mu$ satisfy Assumption~\ref{ass}, and $c\geq 2+K$. Let $(c,u)$, with $u\in C^\infty(\mathbb R)$, be a travelling wave of the Fisher-KPP equation, see \eqref{travelling_wave_KPP}. Let also $\overline w$, $\underline w$ solutions of \eqref{overline_w} and \eqref{underline_w} respectively. Assume $ \underline w(0)\leqslant u(0)\leqslant \overline w(0)$. Then 
 \[ \forall x\leqslant 0, \qquad \underline w(x)\leqslant u(x)\leqslant \overline w(x). \]
 \label{thm:slidingKPPKPP}classical
\end{lem}

The arguments of the proof of Lemma~\ref{thm:slidingsubsupKPP} can be used to prove Lemma \ref{thm:slidingKPPKPP}. We omit the details.

\medskip

%
%
%
%

We can now prove theorem \ref{thm:KPP}.


\begin{proof}[Proof of Theorem \ref{thm:KPP}]
Notice first that $c_*>2+K$, provided $K,\,\mu>0$ are small enough. Let $ \overline w\in C^\infty(\mathbb R)$ and $ \underline w\in C^\infty(\mathbb R)$ satisfying \eqref{overline_w} and \eqref{underline_w} respectively. $\overline w$ and $\underline w$ are then decreasing (see Remark~\ref{rem:w_KPP}), and we may assume (up to a translation) that they satisfy $\underline w(0)=w(0)=u(0)=\overline w(0)$. Then Lemma \ref{thm:slidingsubsupKPP} and \ref{thm:slidingKPPKPP} imply that $\underline w(x)\leqslant w(x), u(x)\leqslant \overline w(x)$ for $x\leq 0$, and then, $\Vert w-u\Vert_{L^\infty(-\infty, 0]}\leqslant \Vert\overline w-\underline w\Vert_{L^\infty(-\infty, 0]}$.

 Let $ \tilde w=\overline w-\underline w\geqslant 0 $, which satisfies
 \[ -c\tilde w'-\tilde w''=\tilde w(1-(\overline w+\underline w))+K+2K\underline w. \]

\medskip

We estimate first the maximum of $ \tilde w $ over $\{x\in\mathbb R;\,\underline w(x)\leqslant 3/4-K\}$ to prove the estimate on
 $ \Vert \tilde w \Vert_{L^\infty(-\infty, 0]}$ stated in Theorem~\ref{thm:KPP}. If $ \underline w\leqslant 3/4-K$, then
 \begin{equation}
 -c\underline w'-\underline w''\geqslant \underline w(1/4-K). \label{eq_esti}
  \end{equation}

 Let
 \[ \lambda_+:=\frac{c+\sqrt{c^2-4(1/4-K)}}{2}, \qquad \lambda_-:=\frac{c-\sqrt{c^2-4(1/4-K)}}{2}, \]
 and $ \varphi(x):=e^{-\lambda_-x}-e^{-\lambda_+x}$. Then $\varphi$ satisfies $-c\varphi'-\varphi''=(1/4-K)\varphi$, $ \varphi(-\infty)=-\infty$ and $ \varphi(+\infty)=0$. Moreover, $ \varphi$ is 
 positive when $ x>0 $ and negative when $ x<0$. Finally, the maximum of $ \varphi$ is attained at $\overline x:=\frac{\ln \lambda_+-\ln\lambda_-}{\lambda_+-\lambda_-}>0$. One can show that $\varphi(\bar x)$ is a continuous and positive function of $c$ and $K$, which is uniformly bounded away from $0$ for $K\in (0,1/8)$ and $c\in[2,\infty)$. There exists thus a universal constant $C>0$ such that $\varphi(\bar x)>C>0$, for any $K\in (0,1/8)$, $c\in[2,\infty)$. Let $\gamma\in(0,3/4-K)$ and $\varphi^\gamma$ defined by
 \[ \varphi^\gamma(x):=\gamma\frac {\varphi(x)}{\varphi(\bar x)},\]
 and $\varphi^\gamma_\eta(x):=\varphi^\gamma(x+\eta)$ for $\eta\in\mathbb R$. Since (for $K>0$ small) $\max_{\mathbb R}\varphi^\gamma\leq 3/4-K<1-2K=\lim_{x\to-\infty}\underline w(x)$ and $\lim_{\eta\to\infty}\varphi^\gamma_\eta(0)=\lim_{x\to\infty}\varphi^\gamma(x)=0<\underline w(0)$, we have that for $\eta>0$ large enough,
 \[ \forall x\leqslant 0,\quad \varphi^\gamma_\eta(x)\leqslant \underline w(x).\]
 Let $ \tilde \eta:=\inf\{\eta\in\mathbb R;\; \forall x\leqslant 0, \,\varphi^\gamma_\eta(x)\leqslant\underline w(x)\}$. Then
 $ \varphi^\gamma_{\tilde\eta}\leq \underline w$ on $ (-\infty, 0]$, and since $\varphi^\gamma_{\tilde \eta}(x)<0$ for $x<<0$, either $\underline w(0)=\varphi^\gamma_{\tilde \eta}(0)$, or there exists $\tilde x\in (-\infty,0)$ such that $\underline w(\tilde x)=\varphi^\gamma_{\tilde \eta}(\tilde x)$. In the latter case,  $\tilde x$ is the minimum of $\underline w-\varphi^\gamma_{\tilde\eta}$, and then
\begin{eqnarray*}
 0&\geq& -c(\underline w-\varphi^\gamma_{\tilde\eta})'(\tilde x)-(\underline w-\varphi^\gamma_{\tilde\eta})''(\tilde x)\\
&=&(1-2K-\underline w(\tilde x))\underline w(\tilde x)-(1/4-K)\varphi^\gamma_{\tilde\eta}(\tilde x)\\
&\geq&\left(3/4-K-\underline w(\tilde x)\right)\underline w(\tilde x)>0,
\end{eqnarray*}
since $\underline w(\tilde x)=\varphi^\gamma_{\tilde\eta}(\tilde x)\leq \gamma<3/4-K$. The above estimate is a contradiction, which implies $\underline w(0)=\varphi^\gamma_{\tilde \eta}(0)$. Then 
 \[ (e^{-\lambda_-\tilde \eta}-e^{-\lambda_+\tilde\eta})=\frac{w(0)}{ \gamma }\left(e^{-\lambda_-\overline x}-e^{-\lambda_+\overline x}\right),\]
and then
\[-\lambda_-\tilde \eta\geq\ln\left(\frac{w(0)}{ \gamma }\left(e^{-\lambda_-\overline x}-e^{-\lambda_+\overline x}\right)\right),\]
which implies $\varphi^\gamma_{\eta}(x)\leq \underline w(x)$ for all $x\in(-\infty,0]$, with $\gamma\in(0,3/4-K)$ and $\eta=-\frac{1}{\lambda_-}\ln\left(\frac{w(0)}{\gamma}(e^{-\lambda_-\overline x}-e^{-\lambda_+\overline x})\right)$. Passing to the limit $\gamma\to 3/4-K$, we then get that $\varphi^{3/4-K}_{\bar \eta}(x)\leq \underline w(x)$ for all $x\in(-\infty,0]$, with
\[ \bar \eta:= -\frac{1}{\lambda_-}\ln\left(\frac{w(0)}{3/4-K}\left(e^{-\lambda_-\overline x}-e^{-\lambda_+\overline x}\right)\right).\]
%
%
In particular, $\varphi^{3/4-K}_{\bar \eta}\leq \underline w$ implies that $\{x\in (-\infty,0];\;\underline w(x)\leqslant 3/4-K\}\subset [\min(0,\overline x-\bar \eta), 0]\subset[\min(0,-\bar \eta), 0]$ (indeed, $-\bar \eta<0$ if $w(0)$ is small enough).

 Since $\sup_{\mathbb R}\underline w=1-2K$, we have
 \[ -c\tilde w'-\tilde w''=\tilde w(1-(\underline w+\overline w))+K+2K\underline w<\tilde w+K(3-4K), \]
 and $ \tilde w(0)=0$. We can then introduce $\psi(x)=K(3-4K)\left(e^{-\alpha x}-1\right)$, with $\alpha=\frac{c-\sqrt{c^2-4}}{2}$ which satisfies $-c\psi'-\psi''=\psi+K(3-4K)$. A sliding argument (that we skip here) shows that
 \[\forall x\leq 0,\quad \tilde w(x)\leqslant \psi(x)= K(3-4K)\left(e^{-\alpha x}-1\right).\]
This estimate implies that
 \[ \underset{[-\bar \eta, 0]}{\max}\tilde w\leqslant K(3-4K)\exp\left(-\frac{\alpha}{\lambda_-}\ln\left(\frac{w(0)}{3/4-K}\left(e^{-\lambda_-\overline x}-e^{-\lambda_+\overline x}\right)\right)\right)\leqslant C\,Kw(0)^{-\frac{\alpha}{\lambda_-}}, \]
 where $ C>0$ is a universal constant. 

\medskip

We consider now the case where the maximum of $\tilde w$ is reached on $[-\infty,0)\setminus\{x\in\mathbb R;\,\underline w(x)\leqslant 3/4-K\}$. If this supremum is a maximum attained in $\bar x$, then $\overline w(\bar x)+\underline w(\bar x)\geq \frac 32-2K>1$ (this last inequality holds if $K$ is small enough), and $-c\tilde w'(\bar x)-\tilde w''(\bar x)\geqslant 0$, which implies
 \[ (\overline w+\underline w-1)\tilde w(\overline x) \leqslant K+2K\underline w\leqslant K(3-4K), \] 
that is $\tilde w(\overline x)\leqslant \frac{K(3-4K)}{1/2-2K}\leq CK$ for some constant $C>0$, provided $K>0$ is small enough. If the supremum is not a maximum, it is possible to obtain a similar estimate, we skip here the additional technical details.

%

\medskip

We have shown that 
\[\sup_{[-\infty,0]}\tilde w\leq \max\left(CK,CKw(0)^{-\frac{\alpha}{\lambda_-}}\right),\]
%
%
We choose now $\beta=(1+\alpha/\lambda_-)^{-1}\in (0,1/2)$ 
 and $w(0)=K^\beta$ (we recall that the solution $(c,w,m)$ is still a solution when $w$ and $m$ are translated). Then, $\sup_{[-\infty,0]}\tilde w\leqslant CK^{\beta}$, and thus
 \[ \Vert w-u\Vert_{L^\infty((-\infty, 0])}\leqslant \Vert\tilde w\Vert_{L^\infty((-\infty, 0])}\leqslant CK^\beta. \]
 Furthermore, $ w $ and $ u $ are decreasing for $ x\geqslant 0$ thanks to Proposition \ref{prop:w_left}, which implies that 
 \[ \forall x\geqslant 0, |w-u|(x)\leqslant w(x)+u(x)\leqslant w(0)+u(0)\leqslant 2K^\beta
.\]
\qg{
From \cite{GT98}, theorem 8.33, there exists a universal constant that we denote $ C>0 $ such that 
\begin{equation}
\Vert v\Vert_{C^{1, \alpha}}\leqslant C,
\end{equation}
where $ v $ is a solution of \eqref{travelling_wave_KPP}, and this constant $ C $ is uniform in the speed $ c$ in the neighbourhood of $ c_0=2\sqrt r. $ In particular, $ u $ satisfies
\begin{equation}
-c_0u'-u''=(c_*-c_0)u'+u(1-u) = u(1-u)+\mathcal O(K).
\end{equation}
Let $ v $ the solution of \eqref{travelling_wave_KPP} with speed $ c_0 $ and $ v(0)=u(0), $ the above argument can then be reproduced to show that 
\begin{equation}
\Vert u-v\Vert_{L^\infty}\leq CK^\beta,
\end{equation}
where $ C $ is a universal constant and $ \beta $ depend only on $ r, $ }
 which finishes the proof.
\end{proof}

\bibliography{MaBibli}


\appendix

\section{Appendix}

\subsection{Compactness results}

We provide here two results that are used in the proof of Theorem~\ref{thm:main}.

\begin{lem}[Elliptic estimates]
 Let $a,\,b^-,\,b^+\in \mathbb R_+^\ast$, $ g\in L^\infty(-a, a)$, and $ \gamma >0. $ For any
 $ b^+, b^-\in\mathbb R $ and 
 $ c\in [-\gamma, \gamma], $ the Dirichlet problem
 
\[ \left\{\begin{array}{l} -cw'-w''=g, \quad (-a, a), \\ w(-a)=b^-,\, w(a)=b^+, \end{array}\right. \]
has a unique weak solution $ w. $ In addition we have $ w\in C^{1, \alpha}([-a, a]) $ for all 
$ \alpha\in [0, 1)$, and there is a constant $ C>0 $
depending only on $ a $ and $ \gamma $
such that 

\[ \Vert w\Vert_{C^{1, \alpha}([-a, a])}\leqslant C(\max(|b^+|,|b^-|)+\Vert g\Vert_{L^\infty}), \]

\label{lem:ellipticestimates}
\end{lem}

\begin{proof}[Proof of Lemma~\ref{lem:ellipticestimates}]
As the domain lies in $ \mathbb R, $ we are not concerned with the regularity 
problem near the boundary. Since
\[ L^\infty(-a, a)\subset \underset{p>1}{\bigcap}L^p(-a, a), \]
theorem 9.16 \cite{GT98} gives us existence and uniqueness of a solution $ w\in W^{2, p}, $ for all 
$ p > 1. $ We deduce from Sobolev imbedding that $ w\in C^{1, \alpha}([-a, a]) $ for all $ \alpha < 1. $ 

The classical theory (\cite{GT98}, theorem 3.7) gives us a constant $ C'>0 $ depending only on $ a $ and $ \gamma $ 
such that
\[ \Vert w\Vert_{L^\infty}\leqslant \max(b^+, b^-)+C'\Vert g\Vert_{L^\infty}. \]
The estimate on the H\"older norm of the first derivative comes now from \cite{GT98}, theorem 8.33, 
which states that whenever $ w $ is a $ C^{1, \alpha} $ solution of $ -cw'-w''=g $ with $ g\in L^\infty, $ 
then
\[ \Vert w\Vert_{C^{1, \alpha}([-a, a])}\leqslant C''(\Vert w\Vert_{C^0([-a, a])}+\Vert g\Vert_{L^\infty}), \]
with a constant $ C''=C''(a, \gamma) $ depending only on $ a $ and $ \gamma. $ That proves the theorem.
\end{proof}

\begin{lem}
Let $a,\,b^-,\,b^+\in \mathbb R_+^\ast$. The operator 
$ (L)^{-1}_D:\mathbb R\times C^0([-a,a]) \longrightarrow C^0([-a,a]) $ 
defined by
 \[ L_D^{-1}(c, g)=w, \]
 where $w$ is the unique solution of
\[ \left\{\begin{array}{l} -cw'-w''=g, \quad (-a, a), \\ 
                           w(-a)=b^+, w(a)=b^-, 
\end{array}\right. \]
is continuous and compact.
\label{lem:degreecompacity}
\end{lem}
\begin{proof}[Proof of Lemma~\ref{lem:degreecompacity}]
 Let $ (c, g), (\tilde c, \tilde g) \in\mathbb R\times C^0([-a,a])$, 
 $ \gamma > 0 $ and $ w, \tilde w\in C^0([-a,a]) $ 
 such that $ c, \tilde c \leqslant \gamma $ and
 
\[ \left\{\begin{array}{l} -cw'-w''=g\textrm{ on }(-a, a), \\ 
                           w(-a)=b^+, w(a)=b^-, 
\end{array}\right. \]
\[ \left\{\begin{array}{l} -\tilde c\tilde w'-\tilde w''=\tilde g\textrm{ on }(-a, a), \\ 
                           \tilde w(-a)=b^+, \tilde w(a)=b^-.
\end{array}\right. \]
Then $ w-\tilde w $ satisfies

\[ \left\{\begin{array}{l} -c(w-\tilde w)'-(w-\tilde w)''=g-\tilde g+(c-\tilde c)\tilde w'
\textrm{ on }(-a, a), \\ 
                           (w-\tilde w)(-a)=0, (w-\tilde w)(a)=0.
\end{array}\right. \]
We deduce from Lemma \ref{lem:ellipticestimates} that there exists a constant $ C $ depending only on 
$ a>0 $ 
 such that
\[ \Vert w-\tilde w\Vert_{C^0}\leqslant C(\Vert g-\tilde g\Vert_{C^0}+|c-\tilde c|
(\Vert\tilde g\Vert_{C^0}+\max(b^+, b^-))), \]
which shows the pointwise continuity of $ L_D^{-1}. $

Now let $ (c_n, g_n) $ a bounded sequence in $ \mathbb R\times C^0. $ Let $ \gamma=\limsup |c_n|. $ From 
Lemma \ref{lem:ellipticestimates}
we deduce the existence of a constant $ C>0 $ depending only on $ a $ and $ \gamma $ such that 
\[ \Vert u_n\Vert_{C^1}\leqslant C(\max(b^+, b^-)+\Vert g_n\Vert_{C^0}), \]
where $ u_n=L_D^{-1}(c_n, g_n), $ which shows that $ (g_n) $ is bounded in $ C^1. $ Since $ C^1 $ is 
compactly embedded in $ C^0, $
there exists a $ w\in C^0 $ such that $ \Vert u_n-w\Vert_{C^0} \rightarrow 0. $ This shows the compactness of 
$ L_D^{-1}. $ 
\end{proof}

\subsection{Properties of the reaction terms}
\label{appendix_reaction_terms}

The proofs of Theorem~\ref{thm:monotonicity} requires precise estimates on the reaction terms $f_w$ and $f_m$. Here we prove a number of technical lemmas that are necessary for our study.

\begin{lem}\label{lem:assumption2}
Let $r,\,K,\,\mu$ satisfy Assumption~\ref{ass}, and $\phi_w$, $\phi_m$ defined by \eqref{def:phiw} and \eqref{def:phim} respectively. Then, $\phi_w,\,\phi_m:\mathbb R_+\to \mathbb R$ are decreasing functions such that
\[\phi_w([0,K])\subset (\mu,\infty),\quad \phi_m([0,1])\subset (\mu K/r,\infty).\]
\end{lem}

\begin{proof}[Proof of Lemma~\ref{lem:assumption2}]
 We prove the lemma for $\phi_m$. The results on $\phi_w$ follow since both functions coincide when $r=K=1$. 
\delbyqg{Notice that the result on $\phi_w$ can be obtained from the one on $\phi_m$ by setting $ r=K=1$. We will thus only consider the case of $\phi_m$.} The fact that $\phi_m$ is decreasing simply comes from the computation of its derivative:
\[\phi_m'(w)=\frac 12\left[-1+\frac{-2\left(K-\frac{\mu K}{r}-w\right)+4\frac{\mu K}{r}}{2\sqrt{\left(K-\frac{\mu K}{r}-w\right)^2+4\frac{\mu K}{r}w}}\right],
\]
one can check that $\phi_m'(w)<0$ for all $w\geq 0$ as soon as $\mu<\frac r2$. Next, we can estimate $\phi_m(w)$ for $w>0$ large:
\begin{eqnarray*}
 \phi_m(w)&=&\frac{w+\frac{\mu K}r-K}2\left(-1+\sqrt{1+\frac{4\mu Kw}{r\left(w+\frac{\mu K}r-K\right)^2}}\right)\\
&=&\frac{w+\frac{\mu K}r-K}2\left(\frac{2\mu Kw}{r\left(w+\frac{\mu K}r-K\right)^2}+o(1/w)\right)\\
&=&\frac{\mu K}r+o(1),
\end{eqnarray*}
that is $\lim_{w\to\infty}\phi_m(w)=\frac{\mu K}r$, which, combined to the variation of $\phi_w$, shows that $\phi_m([0,1])\subset (\mu K/r,\infty)$.

\end{proof}

\begin{lem}
\label{lem:zerosfmfw}
Let $r,\,K,\,\mu$ satisfy Assumption~\ref{ass}, $\phi_w$, $\phi_m$ defined by \eqref{def:phiw} and \eqref{def:phim} respectively.
\[ Z_w=\{(w, m)\in (0, 1)\times(0, K)/f_w(w,m)=0 \}, \]
\[ Z_m=\{(w,m)\in (0, 1)\times(0, K)/f_m(w, m)=0 \}, \]
and denote
\begin{equation}
\mathcal D=(0, 1)\times(0, K). 
\end{equation}
Then:
\begin{enumerate}
\item $ Z_w $ can be described in two ways:
\begin{equation}\label{eq:fwzerosx}
Z_w=\left\{\left(\phi_w(m)
, m\right), m\in(0, K) \right\},
\end{equation}
and
\begin{equation}\label{eq:fwzerosy}
Z_w=\left\{\left(w,\varphi_w(w)
\right),\qg w \in(\mu, 1) \right\} \cap \mathcal D,
\end{equation}
where $\varphi_w(w)=\frac{w(1-\mu-w)}{w-\mu}$.
\item Similarly, $ Z_m $ can be described as:
\begin{equation}
\label{eq:fmzerosx}
Z_m=\left\{\left(w, \phi_m(w)
\right), w\in(0, 1)\right\},
\end{equation}
and
\begin{equation}
\label{eq:fmzerosy}
Z_m=\left\{\left(\varphi_m(m)
, m\right), m\in\left(\frac{\mu K}{r}, K\right) \right\}\cap\mathcal D,
\end{equation}
where $\varphi_m(m):=\frac{m(K-\frac{\mu K}{r}-m)}{m-\frac{\mu K}{r}}$.
\end{enumerate}
\end{lem}
\begin{proof}[Proof of Lemma~\ref{lem:zerosfmfw}]
Notice that point 1 can be obtained from point 2 by setting $ r=K=1. $ Thus, we are only going to prove point 2. We write
\[ f_m(w, m)=rm\left(1-\frac{w+m}{K}\right)+\mu(w-m)=-\frac{r}{K}m^2+\left(r-\mu-\frac{r}{K}w\right)m+\mu w. \]
Since $ \Delta=\left(r-\mu-\frac{r}{K}w\right)^2+4\qg{\frac{\mu r}{K}} w >0 $ for any $ w\geq 0, $ $ f_m(w, m)=0 $ admits only two solutions for $ w\geq 0 $ fixed. Those write:
\[ \frac{1}{2} \left((K-\frac{\mu K}{r}-w)\pm\sqrt{(K-\frac{\mu K}{r}-w)^2+4\frac{\mu K}{r} w}\right),\]
one of those two solutions is negative for all $w\neq 0$, so that $f_m(w,m)=0$ with $(w,m)\in\mathcal D$ implies that $m=\phi_m(w)$,
%
which leads to \eqref{eq:fmzerosx}.

Thanks to Lemma~\ref{lem:assumption2}, $m>\mu K/r$ on $Z_m$, $f(w,m)=0$ with $(w,m)\in\mathcal D$ then implies $w=\varphi_m(m)$.
For $ m\in\left(0, \frac{\mu K}{r}\right), $ $ \varphi_m(m) $ is decreasing and
\[ \left\{\begin{array}{l} \varphi_m(0)=0, \\ \underset{m\rightarrow\left(\frac{\mu K}{r}\right)^-}{\lim\phantom{fsd}}\varphi_m(m)=-\infty, \end{array}\right. \]
so that $ \varphi_m(m)<0 $ for $ m \in \left(0, \frac{\mu K}{r}\right). $ That proves \eqref{eq:fmzerosy}. 
\end{proof}
The next lemma proves that $ f $ admits only one zero in $ \mathcal D, $ and proves some inclusions between $ f_m>0 $ and $ f_w>0. $ 
\begin{lem}
\label{lem:deczeros}
Let $r,\,K,\,\mu$ satisfy Assumption~\ref{ass}, $\phi_w$, $\phi_m$, $\varphi_w$, $\varphi_m$ defined by \eqref{def:phiw}, \eqref{def:phim}, \eqref{eq:fwzerosy} and \eqref{eq:fmzerosy} respectively. Then $ \phi_w $ and $ \varphi_m $ are convex, strictly decreasing functions over $ (0, K) $ and $ \left(\frac{\mu K}{r}, K\right) $ respectively. 
\end{lem}
\begin{proof}[Proof of Lemma~\ref{lem:deczeros}]
We have already shown that $\phi_w$ is decreasing on $(0,K)$. We compute:
\[ \phi_w'(m)=-\frac{1}{2}\left(1+\frac{1-3\mu-m}{\sqrt{(1-\mu-m)^2+4\mu m}}\right).\]
Computing the second derivative, we find:
\[ \phi_w''(m)=\frac{(1-\mu-m)^2+4\mu m-(1-3\mu-m)^2}{2\left((1-\mu-m)^2+4\mu m\right)^{\frac{3}{2}}} \]
\[ =\frac{2\mu(1-2\mu)}{\left((1-\mu-m)^2+4\mu m\right)^{\frac{3}{2}}} >0, \]
so that $ \phi_w $ is convex over $ \mathbb R_+. $ 
Thanks to polynomial arithmetics, we compute:
\[ \varphi_m(m)=\frac{m\left(K\left(1-\frac{\mu}{r}\right)-m\right)}{m-\frac{\mu K}{r}}=K\left(1-\frac{2\mu}{r}\right)-m+\frac{\frac{\mu K^2}{r}\left(1-\frac{2\mu}{r}\right)}{m-\frac{\mu K}{r}}, \]
which makes $ \varphi_m $ obviously convex and strictly decreasing on $ \left(\frac{\mu K}{r}, K\right) $.
\end{proof}

\begin{lem}
\label{lem:zerof}
There exists a unique solution to the problem:
\begin{equation}
\label{eq:pbzerof}
f_w(w, m)=f_m(w, m)=0,
\end{equation}
with $ (w, m)\in(0,1)\times(0, K). $
\end{lem}
\begin{proof}[Proof of Lemma~\ref{lem:zerof}]
We write:
\[ f_w=w(1-\mu-w)+m(\mu-w). \]
Since $ \mu < \frac{1}{2}, $ we have
\[ f_w(\mu, m)=\mu(1-2\mu)>0, \]
so that there cannot be a solution of $ f_w(w, m)=0 $ with $ w=\mu. $ Thus, $ f_w(w, m)=0 $ if and only if
\begin{equation}
\label{eq:mfw}
 m=\frac{w(1-\mu-w)}{w-\mu}.
\end{equation}
Substituting \eqref{eq:mfw} in $ f_m(w, m)=0, $ we get:
\[ \eqref{eq:pbzerof}\Rightarrow\underset{A}{\underbrace{r\frac{w(1-\mu-w)}{w-\mu}\left(1-\frac{\frac{w(1-\mu-w)}{w-\mu}+w}{K}\right)}}+\underset{B}{\underbrace{\mu\left(w-\frac{w(1-\mu-w)}{w-\mu}\right)}}=0. \]
We compute:
\[ A=\frac{rw(1-\mu-w)}{K(w-\mu)^2}\left(w(K+2\mu-1)-\mu K\right), \]
\[ B=\mu \frac{w(2w-1)}{w-\mu}. \]
From now on we assume $ w\neq 0. $ Then:
\[ \eqref{eq:pbzerof}\Rightarrow C(w):=\frac{r}{K}(1-\mu-w)(w(K+2\mu-1)-\mu K)+\mu(2w-1)(w-\mu)=0. \]
Now $ C $ is a polynomial function of degree at most 2. We compute:
\[ C(0)=\mu(1-r(1-\mu))<0, \]
\[ C(1)=\mu\left(1-\mu+\frac{r}{K}((1-\mu)(1-K)-\mu)\right)>0, \]
under the following assumptions:
\[ \mu<1-\frac{1}{r}, \]
\[ K<\frac{r}{r-1}\left(1-\frac{\mu}{1-\mu}\right). \]
That proves the uniqueness of a solution of \eqref{eq:pbzerof} with $ w\in(0, 1). $

Now recall the notations of lemmas \ref{lem:zerosfmfw} and \ref{lem:deczeros}. The existence of a solution to problem \eqref{eq:pbzerof} is equivalent to showing $ Z_m\cap Z_w\neq\varnothing. $ Since:
\[ \Phi_w\left(\frac{\mu K}{r}\right)\in\mathbb R, \]
\[ \underset{m\rightarrow \left(\frac{\mu K}{r}\right)^+}{\lim}\varphi_m(m)=+\infty, \]
\[ \Phi_w(K)=\frac{1}{2}\left(1-\mu-K+\sqrt{(1-\mu-K)^2+4\mu K}\right)>0, \]
\[ \varphi_m(K)=-\frac{\mu}{r-\mu}<0, \]
and since $ \Phi_w $ and $ \varphi_m $ are continuous over $ \left(\frac{\mu K}{r}, K\right), $ there exists a solution to $ \varphi_m(m)=\Phi_w(m) $ with $ m\in \left(\frac{\mu K}{r}, K\right). $ 

Since $ \forall m\in(0, K), 0<\Phi_w(m)<1, $ that gives us a solution to \eqref{eq:pbzerof}, and proves Lemma \ref{lem:zerof}. 
\end{proof}

\begin{lem}\label{lem:zerosinclusions}
Let $ \mathcal D=(0, 1)\times(0, K), $  
\[ \gr{Z_w=\{f_w=0 \}\cap \mathcal D=\{w=\phi_w(m)\}\cap \mathcal D}, \quad Z_w^-=\{f_w<0\}\cap \mathcal D=\{w>\phi_w(m)\}\cap\mathcal D, \]
\[ \gr{Z_m=\{f_m=0 \}\cap \mathcal D=\{m=\phi_m(w)=0 \}\cap \mathcal D},\quad Z_m^-=\{f_m<0\}\cap\mathcal D=\{m>\phi_m(w)<0\}\cap\mathcal D, \]
and
\[ \mathcal D_l=\{(w, m)\in\mathcal D, w\leq w^*, m\geq  m^* \}, \]
\[ \mathcal D_r=\{(w, m)\in\mathcal D, w\geq  w^*, m\leq m^* \}, \]
where $ (w^*, m^*) $ is the only solution of $ f_m=f_w=0 $ in $ \mathcal D. $ 
Then 
\begin{equation}
\label{eq:redzeros}
 Z_m\cup Z_w\subset \mathcal D_l\cup \mathcal D_r.
\end{equation}
Moreover, 
\begin{equation}
\label{eq:zeroinclusionwsm}
Z_w^-\cap \mathcal D_l\subset Z_m^-,
\end{equation}
\begin{equation}
\label{eq:zeroinclusionmsw}
Z_m^-\cap \mathcal D_r\subset Z_w^- .
\end{equation}
\end{lem}

\begin{Rk}\label{Rk:zerosinclusions}
 Let
\[Z_w^+=\{f_w>0\}\cap\mathcal D= \{w<\phi_w(m)\}\cap\mathcal D ,\quad Z_m^+=\{f_m>0\}\cap\mathcal D=\{m<\phi_m(w)\}\cap\mathcal D.\]
Lemma~\ref{lem:zerosinclusions} implies that
\[Z_m^+\cap \mathcal D_l\subset Z_w^+,\quad Z_w^+\cap \mathcal D_r\subset Z_m^+.\]
\end{Rk}

\begin{proof}[Proof of Lemma~\ref{lem:zerosinclusions}]
Assertion \eqref{eq:redzeros} comes from the fact that $ \Phi_w $ and $ \varphi_m $ are decreasing. 

Assertion \eqref{eq:zeroinclusionwsm} comes from the fact that $ \Phi_w-\varphi_m $ is negative for $ m $ close to $ \left(\frac{\mu K}{r}\right)^+ $ and does not change sign in $ \left(\frac{\mu K}{r},  m^*\right) $ since $ (w^*, m^*) $ is the only solution of \eqref{eq:pbzerof}. A similar argument proves assertion \eqref{eq:zeroinclusionmsw}
\end{proof}
The last thing we need here is an estimate of the behaviour of $ m^*(\mu, r, K) $ when $ \mu\rightarrow 0 $:

\begin{lem}
\label{lem:estmu0}
For $ \mu<1-K, $ we have
\begin{equation}
m^*(\mu, r, K)\leq\frac{\frac{\mu K}{r}(1-\mu)}{1-\mu-K\left(1-\frac{2\mu}{r}\right)}.
\end{equation}
\end{lem}
\begin{proof}[Proof of Lemma~\ref{lem:estmu0}]
Recall the notations of Lemma \ref{lem:deczeros}. From Lemma \ref{lem:zerof} we know that $ m^* $ is the only solution of $ \Phi_w=\varphi_m $ that lies in $ (0, K). $ Since $ m\mapsto \sqrt m $ is increasing and $ 1-\mu-K>0, $ we have:
\begin{equation}
\Phi_w(m)\geq 1-\mu-m.
\end{equation}
We deduce then:
\[ \varphi_m(m)-\Phi_w(m)\leq\varphi_m(m)-(1-\mu-m). \]
Now $ \varphi_m-\Phi_w $ is positive near $ \left(\frac{\mu K}{r}\right)^+, $ and for $ w\in \left(\frac{\mu K}{r}, m^*\right), $
\[ 0<\varphi_m(w)-\Phi_w(w)\leq\varphi_m(w)-(1-\mu-m), \]
which means that if $ \bar m $ satisfies 
\begin{equation}
\label{eq:lemestmu0eq1} 
\varphi_m(\bar m)=1-\mu-\bar m. 
\end{equation}
then $ \bar m\geq  m^*. $ 
A simple computation shows that the only solution of \eqref{eq:lemestmu0eq1} is:
\[ \bar m=\frac{\frac{\mu K}{r}(1-\mu)}{1-\mu-K\left(1-\frac{2\mu}{r}\right)}, \]
which finishes to prove Lemma \ref{lem:estmu0}
\end{proof}

\end{document}